\documentclass[11pt]{amsart}
\usepackage{amscd,amsmath,amssymb,amsfonts,verbatim}
\usepackage[cmtip, all]{xy}
\usepackage{MnSymbol}
\usepackage{comment}

\newtheorem{thm}{Theorem}[section]
\newtheorem{prop}[thm]{Proposition}
\newtheorem{lem}[thm]{Lemma}
\newtheorem{cor}[thm]{Corollary}

\theoremstyle{definition}
\newtheorem{ques}[thm]{Question}

\newtheorem{defn}[thm]{Definition}
\theoremstyle{remark}
\newtheorem{remk}[thm]{Remark}
\newtheorem{remks}[thm]{Remarks}

\newtheorem{exm}[thm]{Example}
\newtheorem{exms}[thm]{Examples}
\newtheorem{notat}[thm]{Notation}
\numberwithin{equation}{section}

{\hfill$\square$\end{defn}}
{\hfill$\square$\end{remk}}
{\hfill$\square$\end{remks}}
{\hfill$\square$\end{exm}}
{\hfill$\square$\end{exms}}
{\hfill$\square$\end{notat}}

\newcommand{\thmref}{Theorem~\ref}
\newcommand{\propref}{Proposition~\ref}
\newcommand{\corref}{Corollary~\ref}

\newcommand{\lemref}{Lemma~\ref}

\newcommand{\secref}{Section~\ref}

\newcommand{\sF}{{\mathcal F}}

\newcommand{\sI}{{\mathcal I}}

\newcommand{\sK}{{\mathcal K}}
\newcommand{\sL}{{\mathcal L}}

\newcommand{\sO}{{\mathcal O}}

\newcommand{\sR}{{\mathcal R}}

\newcommand{\sU}{{\mathcal U}}

\newcommand{\sZ}{{\mathcal Z}}

\renewcommand{\P}{{\mathbb P}}

\newcommand{\Z}{{\mathbb Z}}

\newcommand{\fm}{{\mathfrak m}}

\newcommand{\Ker}{{\rm Ker}}

\newcommand{\CH}{{\rm CH}}

\newcommand{\surj}{\twoheadrightarrow}
\newcommand{\inj}{\hookrightarrow}
\newcommand{\red}{{\rm red}}

\newcommand{\Spec}{{\rm Spec \,}}
\newcommand{\sing}{{\rm sing}}

\newcommand{\divf}{{\rm div}}

\newcommand{\Sch}{{\operatorname{\mathbf{Sch}}}}

\newcommand{\Sm}{{\mathbf{Sm}}}

\newcommand{\cyc}{{\operatorname{\rm cyc}}}

\newcommand{\ds}{{/\kern-3pt/}}

\newcommand{\ov}{\overline}

\renewcommand{\dim}{\text{\rm dim}}

\newcommand{\tuborg}{\left\{\begin{array}{ll}}
\newcommand{\sluttuborg}{\end{array}\right.}

\newcommand{\zar}{{\rm zar}}
\newcommand{\nis}{{\rm nis}}

\newcommand{\reg}{{\rm reg}}

\newcommand{\wt}{\widetilde}

\newcounter{elno}

\newcounter{elno-abc}   
\newenvironment{listabc}{
                         \begin{list}{\alph{elno-abc})
                                     }{\usecounter{elno-abc}}
                      }{
                         \end{list}}
\newcounter{elno-abc-prime}   
\newenvironment{listabcprime}{
                         \begin{list}{\alph{elno-abc-prime}')
                                     }{\usecounter{elno-abc-prime}}
                      }{
                         \end{list}}

\begin{document}
\title{A decomposition theorem for 0-cycles and applications}
\author{Rahul Gupta, Amalendu Krishna, Jitendra Rathore}
\address{Fakult\"at f\"ur Mathematik, Universit\"at Regensburg, 
93040, Regensburg, Germany.}
\email{Rahul.Gupta@mathematik.uni-regensburg.de}
\address{Department of Mathematics, Indian Institute of Science,  
Bangalore, 560012, India.}
\email{amalenduk@iisc.ac.in}
\address{School of Mathematics, Tata Institute of Fundamental Research, Homi Bhabha
  Road, Mumbai-400005, India.}
\email{jitendra@math.tifr.res.in}




\maketitle

\begin{quote}\emph{Abstract.} 
  We prove a decomposition theorem for the cohomological Chow group of 0-cycles on the
  double of a quasi-projective $R_1$-scheme over a field along a closed subscheme, in
  terms of the Chow groups, with and without modulus, of the scheme.
  This yields a significant generalization of the decomposition theorem of
  Binda-Krishna. As applications, we prove a moving lemma for Chow groups with modulus
  and an analogue of
  Bloch's formula for 0-cycles with modulus on singular surfaces. The latter extends a
  previous result of Binda-Krishna-Saito.
  \end{quote}
\setcounter{tocdepth}{1}
\tableofcontents

\section{Introduction}\label{sec:Intro}
\subsection{Background}\label{sec:Background}
Let $k$ be a field and $X$ an integral quasi-projective $R_1$-scheme
(i.e., a scheme which is  regular in codimension one) of dimension $d \ge 1$ over $k$.
Let $D \subset X$ be a closed subscheme of 
codimension one. Let $S_X$ denote the join of two copies of $X$ along $D$ (see
\S~\ref{sec:double}).
Let $\CH_0(X|D)$ denote the Chow group of 0-cycles with modulus
on $X$ (see \cite[Definition~1.6]{Kerz-Saito})
and let $\CH_0(S_X)$ denote the cohomological (lci) Chow group of 0-cycles on
$S_X$, introduced in \cite{Binda-Krishna}. Let $\CH^F_0(X)$ denote the classical
(Borel-Moore type) homological
Chow group of 0-cycles on $X$ (see \cite[Chapter~1]{Fulton}).

When $k$ is perfect and $X$ is smooth over $k$,
a fundamental result of Binda-Krishna \cite{Binda-Krishna}
connecting 0-cycles with and without modulus says that there exists a split
short exact sequence
\begin{equation}\label{eqn:BK}
  0 \to \CH_0(X|D) \to \CH_0(S_X) \to \CH^F_0(X) \to 0.
\end{equation}

This result plays a fundamental role in the current research on 0-cycles with
modulus on smooth varieties over perfect fields. Several new results have been found in the
subject of cycles with modulus using this exact sequence.
One reason for its usefulness is that it allows one to transport any
problem about 0-cycles with modulus on smooth schemes
to an analogous problem about 0-cycles without
modulus on schemes, albeit with mild singularities. In many cases, the latter
groups are well understood.

\subsection{Objective}
The goal of this paper is to prove an analogue of ~\eqref{eqn:BK}
when neither $k$ is perfect nor $X$ is smooth over $k$.
The motivation for this generalization comes from our desire to: (1) extend existing
results on 0-cycles with modulus on smooth schemes over perfect field to
those over imperfect fields, and (2) extend existing results on 0-cycles on
smooth schemes over perfect fields to singular schemes over such fields.

In this paper, we apply this generalization to prove
a new moving lemma for the Chow group of 0-cycles on
regular schemes, prove Bloch's formula for 0-cycles with modulus on
singular schemes, and prove a moving lemma for 0-cycles on singular surfaces.
Further important applications are obtained in \cite{Krishna-SH} (where
the 0-th Suslin homology is identified with the Chow group with 0-cycles with modulus
under some necessary conditions), and \cite{GKR}
(where a class field theory over finite fields is established for schemes with
singularities).

\subsection{Main result}
Let $k$ now be arbitrary and $X$ an integral quasi-projective $R_1$-scheme of dimension $d \ge 1$ over $k$.
Let $X^o$ denote the regular locus of $X$. Note that $X^o$ may not be smooth over $k$
because the latter is not assumed to be perfect.
Let $D \subset X$ be such that $D \cap X^o$ is a Cartier
divisor (this may be empty).
We shall introduce a new 0-cycle group with modulus on $X$ which we shall
denote by $\CH_0(X|D)$, following the classical notation.
It is an extension of the lci version $\CH_0(X)$ of the Levine-Weibel Chow group
of $X$ to the modulus world. Our motivation behind the introduction of
this group is the expectation that this would be useful in understanding the
already existing Chow group with modulus on non-proper smooth varieties, e.g.,
the regular loci of singular projective varieties.
(see, for instance, \cite{GKR} whose main results are proven using this new Chow group
with modulus). 
Let $\CH_0(X,D)$ denote the lci Chow group of $X$ relative to $D \cup X_\sing$
(see \cite[\S~3.3]{Binda-Krishna}).
The new decomposition theorem reads as follows.

\begin{thm}\label{thm:Main-1}
  There is a split short exact sequence
  \[
    0 \to \CH_0(X|D) \to \CH_0(S_X) \to \CH_0(X, D) \to 0.
  \]
  \end{thm}

  \vskip .3cm

\vskip .3cm

When $X$ is smooth over $k$, the group
  $\CH_0(X,D)$ coincides with $\CH^F_0(X)$ as a consequence of
  the moving lemma for higher Chow groups by Bloch
  (see \cite[Lemma~3.20]{Binda-Krishna}). Hence, \thmref{thm:Main-1}
  is identical to ~\eqref{eqn:BK} in this case.
  The group $\CH_0(X,D)$ coincides with $\CH_0(X)$ for many singular
  schemes as well (see \corref{cor:Main-3-0}).

 \begin{remk}\label{remk:Imperfect}
  We emphasize to the reader that even when $X$ is smooth over the base field $k$,
  \thmref{thm:Main-1} asserts more than the decomposition theorem of
  Binda-Krishna in that it does not assume $k$ to be perfect.
  This is achieved by virtue of some  delicate new moving techniques
  (see \S~\ref{sec:double}) 
  and new Bertini theorems proven recently by Ghosh-Krishna
  \cite{Ghosh-Krishna-Bertini}.
  \end{remk}

\subsection{Application I : moving lemma for Chow group with modulus}
  \label{sec:MLCM}
  The proof of \thmref{thm:Main-1} has the following consequence for the Chow group of
  0-cycles with modulus ({\`a} la Kerz-Saito) on regular schemes.
  Let us assume in \thmref{thm:Main-1} that $X$ is regular and $D$ is reduced.
  Then we prove (a stronger version of)  the following new moving lemma for $\CH_0(X|D)$.

  \begin{cor}\label{cor:ML-CH-mod}
    In the definition of $\CH_0(X|D)$, it suffices to consider only those
    integral curves $C \subset X$ not contained in $D$ for which we can 
    assume that $C \cap D_\sing = \emptyset$ and $\nu^*(D)$ is regular,
    where $\nu \colon C_n \to X$ is the canonical map from the normalization of $C$.
    If $k$ is algebraically closed and $X$ is affine or projective,
    we can further assume that $C$ is regular along $D$.
  \end{cor}

  \begin{remk}\label{remk:Suslin-hom}
    One may recall that if $X$ is assumed to be complete in \corref{cor:ML-CH-mod}
    ($D$ still reduced), then there is a canonical surjection
    $\phi_{X|D} \colon \CH_0(X|D) \surj H^{sus}_0(X \setminus D)$, where
    the latter is the 0-th Suslin homology (see \cite[\S~2.2]{Binda-Krishna-EP} or
    \cite[\S~8]{Ghosh-Krishna-CFT}). The reader is now warned that the above corollary
    does not imply that $\phi_{X|D}$ is an isomorphism. The question as to when
    the latter happens, is the subject of \cite{Krishna-SH}.
\end{remk}

\subsection{Application II: Bloch's formula and moving lemma for Chow group of
    singular surfaces}\label{sec:BF}
  Let us assume in \thmref{thm:Main-1} that $\dim(X) = 2$.
  Let $\sK^M_{2, (X,D)}$ be the kernel of the restriction map of Zariski sheaves
    of Milnor $K$-groups $\sK^M_{2,X} \surj \sK^M_{2,D}$ on $X$. Every closed point
    $x \in X^o \setminus D$ has a cycle class $cyc_X([x]) \in
    H^2_\zar(X, \sK^M_{2,(X,D)})$ (e.g., see \cite[\S~2]{Kato-Saito-86}). 
This defines the well known cycle class map
    $\sZ_0(X^o \setminus D) \to H^2_\zar(X, \sK^M_{2,(X,D)})$.
    As another application of \thmref{thm:Main-1}, we are able to prove the following
    Bloch's formulas for Chow groups, with and without modulus, for surfaces over
    arbitrary fields.

    \begin{thm}\label{thm:Main-3}
      The cycle class map induces isomorphisms
      \[
        \cyc_{X} \colon \CH_0(X, D) \xrightarrow{\cong} H^2_\zar(X, \sK^M_{2,X});
      \]
      \[
        \cyc_{X|D} \colon \CH_0(X|D) \xrightarrow{\cong} H^2_\zar(X, \sK^M_{2,(X,D)}).
      \]
    \end{thm}

 Binda-Krishna-Saito \cite{BKS} had earlier proven the first isomorphism when
    $k$ is perfect and $D = \emptyset$ and the second isomorphism when $k$ is perfect
    and $X$ is smooth over $k$.

 One consequence of Bloch's formula is the following moving lemma for
    $\CH_0(X)$.
    \begin{cor}\label{cor:Main-3-0}
     Let $D \subset X$ be as in \thmref{thm:Main-1}.  
      Then the canonical map
      \[
        \CH_0(X,D) \to \CH_0(X)
      \]
      is an isomorphism in the following cases:
      \begin{enumerate}
      \item
        $X$ is smooth over $k$.
      \item
        $d = 2$.
      \item
        $k$ is algebraically closed and $X$ is affine.
        \end{enumerate}
    \end{cor}

    \begin{remk}\label{remk:Levine-unpublished}
      We shall show in the proof of \corref{cor:Main-3-0} that its part (3) also holds
      if $X$ is projective, assuming a result from the yet unpublished paper
      \cite{Levine-unpub}.
      \end{remk}

Note that the case (1) of the corollary is already known by
    \cite[Lemma~3.20]{Binda-Krishna}.
    So the new results are the other two cases.
The case (2) was earlier proven by Levine-Weibel 
when $k$ is algebraically closed and $X$ is affine 
(see \cite[Theorem~2.3]{Levine-Weibel}).

The main application of \corref{cor:Main-3-0} for us is that we can replace
$\CH_0(X,D)$ by $\CH_0(X)$ in \thmref{thm:Main-1} when $X$
satisfies one of the three conditions of the corollary.
We expect this is to be true in general
(see Question~\ref{ques:ML-LW-0}) but do not know how to prove it.

\subsection{Overview of proofs}\label{sec:Outline}
We now provide a brief outline of our proofs.
The proof of \thmref{thm:Main-1} presented here was undoubtedly inspired by the
previous decomposition theorem of Binda-Krishna \cite{Binda-Krishna}.
However, we had to make some critical advances in order to prove the new
decomposition theorem.

The first is to devise methods to get rid of the assumption that the base field $k$ is
perfect.
Due to potential applications, some experts had asked the second author if this could be
possible.
In order to do so, we significantly refine the techniques of \cite{Binda-Krishna}.
This involves developing some delicate new moving techniques. We expect these new moving
techniques to be useful in future study of cycles over imperfect fields.
In particular, these also allow us to clarify some obscure steps in the proof of
the decomposition theorem of \cite{Binda-Krishna}.
Besides these movings, we employ the new Bertini
theorems proved recently by Ghosh-Krishna \cite{Ghosh-Krishna-Bertini} over
arbitrary fields. 

The second is to weaken the assumption that $X$ is smooth over the base field.
Filling this part is very challenging and requires one to make very fine
choices of various intermediate hypersurface sections and curves that we
consider during the proofs. In particular, the reduction of the higher dimensions
case to surfaces becomes much subtle. \S~\ref{sec:double} provides many of
the new technical steps we need in order to achieve this reduction. 

In \S~\ref{sec:LWCM}, we review the definitions of the Levine-Weibel and
lci Chow groups of singular schemes. We then introduce the Chow group of
0-cycles with modulus on singular schemes.
In \S~\ref{sec:CC}, we prove some results which show that the Chow group of
the double of a singular scheme along a closed subscheme has a very special
presentation. This is then used in \S~\ref{sec:PB} to achieve the most
challenging piece of the proof of  \thmref{thm:Main-1}. We also prove
a moving lemma in this section.
In \S~\ref{sec:ES}, we complete the proofs of \thmref{thm:Main-1} and its
applications.

\subsection{Notations}\label{sec:Notn}
Given a field $k$, a $k$-scheme will
mean a separated and essentially of finite type $k$-scheme.
We shall denote the category of such schemes by $\Sch_k$.
We let $\Sm_k$ denote the category of smooth essentially of finite type $k$-schemes.
A curve in this paper will mean a $k$-scheme of pure dimension one.

If $Y,Z \subset X$ are two closed subschemes, then $Y \cap Z$ will mean
the scheme theoretic intersection $Y \times_X Z$ unless we say otherwise.
If $X$ is reduced, we shall let $X_n$ denote the normalization of $X$.
We shall let $k(X)$ denote the total ring of quotients of $X$.
For a morphism $f \colon X' \to X$ of schemes and
$D \subset X$ a subscheme, we shall write $D \times_X X'$ as $f^*(D)$.
We shall let $\sZ_0(X)$ denote the free abelian group on the set of closed points on
$X$.

\section{Levine-Weibel Chow group with modulus}\label{sec:LWCM}
We fix a field $k$ of any characteristic.
We begin by very briefly recalling the definitions of the Levine-Weibel Chow group
and its improved version, the lci Chow group.
We refer to \cite[\S~3]{Binda-Krishna} for complete definitions.
If $X$ is a reduced connected $k$-scheme of dimension $d$, then a point $x \in X$
will be called a singular point if either it lies in an irreducible component of $X$
of dimension $< d$ or $\sO_{X,x}$ is not a regular local ring.
We let $X_\sing$ denote the locus of singular points of $X$. This is a closed
subset of $X$ and we endow it with the reduced closed subscheme structure.
We let $X_\reg = X \setminus X_\sing$. If $X$ is not necessarily reduced, we
shall let $X^\dagger = (X_\red)_\reg$.

Throughout this paper, we shall use the following notations.
Once we have fixed a reduced 
 scheme $X \in \Sch_k$, we shall write $X_\reg$ as $X^o$.
For any subscheme $Z \subset X$, we shall write $Z^o = Z \cap X^o$
and $Z^\star = Z \cup X_\sing$.
Keep in mind that $Z^o \neq Z_\reg$ in general unless $Z = X$.

\subsection{Review of Levine-Weibel and lci Chow groups}\label{sec:LWC}
Let $X$ be a reduced and connected quasi-projective $k$-scheme of dimension $d \ge 1$.
Let $Y \subset X$ be a closed subscheme which contains no
component of $X$ of maximal dimension.

Let $C \subset X$ be a curve. Recall (see \cite[\S~1]{Levine-Weibel} and
\cite[\S~2]{Biswas-Srinivas}) that $C$ is called Cartier relative to $Y$
if no component of $C$ lies in $Y^\star$, no embedded point of $C$ lies away from
$Y^\star$,  $\sO_{C, \eta}$ is a field if $\ov{\{\eta\}}$ is a component of $C$
disjoint from $Y^\star$ and, $C$ is defined by a regular sequence at every point of
$C \cap Y^\star$. If $Y \subset X_\sing$,
we shall simply say that $C$ is a Cartier curve on $X$.
We let $k(C, C\cap Y^\star)^{\times}$ be the image of the natural map
$\sO^{\times}_{C,S} \to \stackrel{s}{\underset{i =1}\oplus} \sO^{\times}_{C, \eta_i}$,
where $\{\eta_1, \ldots , \eta_s\}$ is the set of generic points of $C$ and
$S$ is the union of the closed subset 
$C \cap Y^\star$ and the set of generic points $\eta_i$ of $C$ 
such that $\ov{\{\eta_i\}}$ is 
disjoint from $Y^\star$.

For $f \in k(C, C\cap Y^\star)^{\times}$, we let $\divf(f) =
\stackrel{s}{\underset{i =1}\sum} \divf(f_i)$, where
$f_i$ is the projection of $f$ onto $\sO^{\times}_{C, \eta_i}$, and
$\divf(f_i)$ is the divisor of the restriction of $f_i$ to the maximal Cohen-Macaulay 
subscheme $C_i$ of $C$ supporting $\eta_i$ (see \cite[\S~2]{Biswas-Srinivas}).
If $C$ is reduced, then $k(C, C\cap Y^\star)^{\times} = \sO^{\times}_{C,S}$ and
for $f \in \sO^{\times}_{C,S}$, $\divf(f)$  is the sum of
$\divf(f_i)$, where the sum runs through
the divisors (in the classical sense)
of the restrictions of $f$ to the components of $C$.
If $f$ is a rational function on two different curves $C$ and $C'$, we may
sometimes use notations such as $\divf(f)_C$ and $\divf(f)_{C'}$ to make
clear as to where we consider the divisor of $f$.

Let $A \subset X$ be a reduced closed subscheme such that
$\dim(A) \le d-2$. We let $\sR^{LW}_0(X,Y;A)$ denote the subgroup of
$\sZ_0(X \setminus (Y^\star \cup A))$
 generated by $\divf(f)$, where $f \in k(C, C\cap Y^\star)^{\times}$ for
a Cartier curve $C$ relative to $Y$ which misses $A$.
The Levine-Weibel Chow group of $(X, Y;A)$ is defined as the quotient
\begin{equation}\label{eqn:LW-00}
\CH^{LW}_0(X,Y;A) = \frac{\sZ_0(X \setminus (Y^\star \cup A))}{\sR^{LW}_0(X,Y;A)}.
\end{equation}
We shall drop $Y$ from our notations if $Y \subset X_\sing$.
Thus, the groups $\sR^{LW}_0(X, X_\sing;A)$ and $\CH^{LW}_0(X,X_\sing;A)$ will be simply written
as $\sR^{LW}_0(X;A)$ and $\CH^{LW}_0(X;A)$, respectively.
When $A = \emptyset$, we drop $A$ from our notations of the Chow groups.
Similar convention will be enforced on the lci Chow groups that we recall now.

One says that $C$ a good curve relative to
$Y$ if $C$ is reduced and there is a finite local complete intersection (lci)
morphism $\nu \colon C \to X$ such that $\nu^{-1}(Y^\star)$ is nowhere dense in $C$.
If $Y \subset X_\sing$, we drop the adjective `relative to $Y$'
and simply say that $C$ is a good curve.
We let  $\sR_0(X,Y;A)$ be the subgroup of $\sZ_0(X \setminus (Y^\star \cup A))$
generated by $\nu_*(\divf(f))$, where
$\nu \colon C \to X$ is a good curve relative to $Y$ whose image misses $A$ and
$f \in k(C, \nu^{-1}(Y^\star))^{\times}$.
The lci Chow group of 0-cycles of $(X, Y;A)$ is defined as the quotient 
\begin{equation}\label{eqn:LW-01}
\CH_0(X,Y;A) = \frac{\sZ_0(X \setminus (Y^\star \cup A))}{\sR_0(X,Y;A)}.
\end{equation}
We have a canonical surjection
\[
\CH^{LW}_0(X,Y;A) \surj \CH_0(X,Y;A).
\]

\subsection{Chow groups with modulus}\label{sec:LWGM}
Let $X \in \Sch_k$ be an integral scheme of dimension $d \ge 1$ which is
regular in codimension one.
Let $D \subset X$ be a closed subscheme such that $D^o$ is a Cartier divisor on $X^o$.
We fix a reduced closed subscheme $A \subset X$ such that
$X_\sing \subset A$ and $\dim(A) \le d-2$. 
Let $C$ be an integral normal curve.
We shall say that a morphism $\nu \colon C \to X$ is `admissible' if
it is finite, $\nu(C) \not\subset D$ and $\nu(C) \cap A = \emptyset$.
Given such a morphism and $f \in k(C)^{\times}$, we say that $f$ has modulus $D$
if it lies in $\Ker(\sO^{\times}_{C, E} \surj \sO^{\times}_E)$, where $E = \nu^*(D)$.
We let $\sR_0(X|D;A)$ be the subgroup of $\sZ_0(X^o \setminus (D \cup A))$
generated by $\nu_*(\divf(f))$, where $\nu \colon C \to X$ is admissible and
$f \in k(C)^{\times}$ has modulus $D$.
We let
\begin{equation}\label{eqn:LW-mod-0}
  \CH_0(X|D;A) = \frac{\sZ_0(X^o \setminus (D \cup A))}{\sR_0(X|D;A)}.
\end{equation}

When $A = X_\sing$, we shall often write $\sR_0(X|D;A)$ and $\CH_0(X|D;A)$ as
 $\sR_0(X|D)$ and $\CH_0(X|D)$, respectively. 
Note that $\CH_0(X|D)$ coincides with the original definition
(see \cite[Definition~1.6]{Kerz-Saito}) of the Chow group with modulus when
$X \in \Sm_k$. 
We let $\CH^{\rm mod}_0(X|D;A)$ be the quotient of $\sZ_0(X^o \setminus (D \cup A))$
by the subgroup $\sR^{\rm mod}_0(X|D;A)$, where we allow only those admissible
morphisms $\nu \colon C \to X$ which have the property that $\nu^*(D^\dagger)$
is regular.
We have then the canonical surjection
\begin{equation}\label{eqn:LW-mod-00}
  \CH^{\rm mod}_0(X|D;A) \surj  \CH_0(X|D;A).
\end{equation}

We shall show that this map is in fact an isomorphism (see \corref{cor:mL-lci}).

\vskip .3cm

In this paper, we shall consider a Levine-Weibel analogue of $\CH_0(X|D;A)$.
We let  $\sR^{LW}_0(X|D;A)$ be the subgroup of $\sZ_0(X^o \setminus (D \cup A))$
 generated by $\divf(f)$, where:
\begin{enumerate}
\item
$C \subset X$ is an integral curve with the property
that $C \not\subset D, \ C \cap A = \emptyset$;
\item
  $C$ is regular at every point of $E := C \cap D$;
\item
  $C \cap D^\dagger$ is regular;
\item
  $f \in \Ker(\sO^{\times}_{C,E} \surj \sO^{\times}_E)$.
\end{enumerate}
Note that the conditions (1) and (2) together
force $C$ to be a Cartier curve on $X$ relative to $D \cup A$. We let 
\[
\CH^{LW}_0(X|D;A)=\frac{\sZ_0(X^o \setminus (D \cup A))}{\sR^{LW}_0(X|D;A)}.
\]
We shall write $\CH^{LW}_0(X|D;X_\sing)$ as $\CH^{LW}_0(X|D)$.
We shall call $\CH^{LW}_0(X|D;A)$ the `Levine-Weibel Chow group with modulus' of
$(X,D)$. If $X$ is projective over $k$, we let $\CH_0(X|D;A)_0$ denote the
kernel of the degree map $\CH_0(X|D;A) \to \Z$.

\begin{remk}\label{remk:LW-properties}
  We remark that just like the Levine-Weibel Chow group without modulus,
  $\CH^{LW}_0(X|D;A)$ may not have
  good functorial properties such as finite push-forward
  (except for closed immersions). On the other hand,
it has some advantages over $\CH_0(X|D;A)$. For instance, it admits
cycle class map to $H^d_\nis(X, \sK^M_{d,(X,D)})$ when $k$ is infinite.
At any rate, there is always a canonical surjection
\begin{equation}\label{eqn:Mod-old}
  \phi_{X|D} \colon  \CH^{LW}_0(X|D;A) \surj   \CH_0(X|D;A),
\end{equation}
which we expect to be an isomorphism
when $k$ is infinite. We shall show in this
paper that this is indeed the case when $d = 2$ (see Remark~\ref{remk:Surf-**}).
\end{remk}


\begin{remk}
For a reduced closed subscheme $A \subset X$ such that
$X_\sing \subset A$ and $\dim(A) \le d-2$, we have a canonical map
\[
\CH_0(X|D, A) \to \CH_0(X|D).
\]
For a general $A$, we can not say much about this homomorphism. But it follows from \corref{cor:mL-lci} that the above map is an isomorphism when $A \subset X_{\sing} \cup D_{\red}$. Indeed, in the corollary, we shall prove the stronger claim that the composition of the maps 
\[
 \CH^{\rm mod}_0(X|D;A) \surj \CH_0(X|D, A) \to \CH_0(X|D).
\]
is an isomorphism. A similar result for the Levine-Weibel Chow groups with modulus can be deduced from 
\corref{cor:ML-0}. 
\end{remk}

\subsection{The double and its Cartier curves}\label{sec:double}
 Let $D \subset X$ be as in \S~\ref{sec:LWGM}.
 Recall from \cite[\S~2.1]{Binda-Krishna}
  that the double (or join) of $X$ along $D$ is the push-out $S(X,D):= X \amalg_D X$.
  We shall also write $S(X,D)$ as $S_X$. One knows that
  \begin{equation}\label{eqn:Double-0}
    \xymatrix@C.8pc{
      D \ar[r]^-{\iota} \ar[d]_-{\iota} & X_+ \ar[d]^-{\iota_+} \\
      X_- \ar[r]^-{\iota_-} & S_X}
  \end{equation}
  is a bi-Cartesian square. Moreover, there is a finite morphism
  $\Delta \colon S_X \to X$ whose composition with $\iota_\pm$ is identity.
  We shall use the following result (see \cite[Proposition~2.4]{Binda-Krishna})
  frequently in this paper.
 \begin{prop}\label{prop:double*}
Suppose that $D$ is a Cartier divisor on $X$. Then the following
hold:
\begin{enumerate}
\item
$\Delta$ is finite and flat.
\item
$S_X$ is Cohen-Macaulay if $X$ is so.
\item
  If $f\colon Y \to X$ is lci 
  along $D$, then $Y {\underset{X}\times} S_X \to S_X$ is lci along $D$.
\item
  If $f\colon Y \to X$ is a morphism such that $f^*(D)$ is a Cartier divisor on $Y$,
  then the canonical map $S_Y \to Y {\underset{X}\times} S_X$ is an isomorphism.
\end{enumerate}
\end{prop}

If $C_\pm$ are two closed subschemes of $X$ such that $C_+ \cap D = C_- \cap D$ as
  closed subschemes, then the join $C_+ \amalg_D C_-$ along $D \cap C_\pm$ is
  canonically a closed subscheme of $S_X$.
While proving the main results, we shall often need to recognize Cartier curves on
$S_X$. The aim of the rest of this subsection is to prove some results in this
direction. The main results that we shall use later are \lemref{lem:Cartier-higher} and
\propref{prop:Cartier-higher-*}. 
We fix a field $k$.

\begin{lem}\label{lem:Cartier}
Let $X \in \Sch_k$ be an integral scheme of dimension $d \ge 1$.
Let $D \subset X$ be an effective Cartier divisor on $X$.
Suppose that $C_\pm\subset X$ are two Cartier divisors such that 
$C_\pm \cap D$ is a Cartier divisor on $C_\pm$ and $C_+ \cap D = C_- \cap D$. 
Then $C_+ \amalg_D C_-$ is a Cartier divisor on $S_X$.
\end{lem}
\begin{proof}
Since this is a local question, we can assume that $X = \Spec(A)$ is local and 
$D = V((f)),  C_\pm = V((a_\pm))$ are all principal Cartier divisors on $X$.  
We are given that $\{a_\pm, f\}$ are regular sequences in the local ring $A$
and $(a_+,f) = (a_-,f)$.
Recall here that every permutation of a regular sequence in a 
Noetherian local ring is also a regular sequence.
We can write $a_- = \alpha a_+ + \beta f$. Note that $\alpha \neq 0$
by our assumption. 

We claim that $\alpha \in A^{\times}$. Suppose on the contrary that $\alpha
\in \fm$, where $\fm$ is the maximal ideal of $A$.
Let $\ov{a}$ denote the residue class of an element $a \in A$ in the quotient ring
$A' = A/{(f)}$. Then we get $\ov{a_-} \in \fm (\ov{a_+})$. But this is not 
possible since $(\ov{a_+}) = (\ov{a_-})$ and $\ov{a_\pm}$ is not a zero-divisor in $A'$.

By replacing $a_+$ by $\alpha a_+$, we can assume that
$C_\pm = V((a_\pm))$ and $a_- = a_+ + \beta f$ for some $\beta \in A$.
We now recall that there is an exact sequence of $R$-modules
\begin{equation}\label{eqn:Cartier-0}
0 \to R \to A \times A \xrightarrow{\phi} A' \to 0,
\end{equation}
where $R$ is the coordinate ring of $S_X$ and $\phi((a, b)) = \ov{a} - \ov{b}$. 
It follows that $(a_+, a_-) \in R$. Let $R'$ be the coordinate ring of 
$C_+ \amalg_D C_-$ and let $I = \Ker(R \surj R')$.
We then have an exact sequence
\begin{equation}\label{eqn:Cartier-1}
  0 \to I \to (a_+)  \times (a_-) \xrightarrow{\phi}
  \frac{(a_\pm)}{(a_\pm) \cap (f)} \to 0.
\end{equation}

It is clear that $a:= (a_+, a_-) \in I$. We now claim that $I = (a)$.
Let $(x_+, x_-) \in I$ be any element. This means that
$x_\pm = \alpha_\pm a_\pm$ for some $\alpha_\pm \in A$ and $x_+- x_- \in (f)$.
We need to show that $\alpha_+ - \alpha_- \in (f)$ so that
$(\alpha_+, \alpha_-) \in R$.

To show this, we write
$\alpha_+ a_+ - \alpha_- a_- = \gamma f$ for some $\gamma \in A$.
This implies that
$\alpha_+ a_+ - \alpha_-(a_+ + \beta f) = \gamma f$.
Equivalently, $(\alpha_+ - \alpha_-)a_+ \in (f)$. But this forces
$\alpha_+ - \alpha_-$ to lie in $(f)$ because $\ov{a_+}$ is not a zero-divisor in $A'$.
We have thus shown that $I = (a)$. It is clear from ~\eqref{eqn:Cartier-0} that
$a$ is not a zero-divisor in $R$ because $a_\pm$ is not a zero-divisor in $A$.
The lemma is now proven.
\end{proof}

To find Cartier curves on the double when $d \ge 3$, we need some 
results. We begin with the following lemma in local commutative algebra.

\begin{lem}\label{lem:Algebra-0}
Let $A$ be a Noetherian local ring. Let $I = (a_1, \ldots , a_r)$ and $J$ be two ideals
of $A$ such that $A/J$ is a discrete valuation ring and $\{a_1, \ldots , a_r\}$ is
a regular sequence in $A$. We can then find a regular sequence
$\{b_1, \ldots , b_r\}$ in $A$ such that $I =  (b_1, \ldots , b_r)$ and
$b_i \in I \cap J$ for $i > 1$.
\end{lem}
\begin{proof}
  In the proof, we shall use (as before)
  the fact that every permutation of a regular sequence in a 
Noetherian local ring is also a regular sequence.
We can assume (after a permutation) that $a_1 \notin J$ because there is
nothing to prove otherwise.
Suppose that $a_i \notin J$ for $1 \le i \le r'$ and $a_i \in J$ for $i > r'$.
If $r' = 1$, then we can take $b_i = a_i$ for $i \ge 1$ and we are done.
We therefore assume that $r' >1$.

Let $\ov{a}$ denote the residue class of any element $a \in A$ in the quotient
$A/J$. Let $\pi \in A$ be an element such that $\ov{\pi}$ is a uniformizing
parameter for the dvr $A/J$.
For $i \le r'$, we can then write
$a_i = a'_i + u_i\pi^{m_i}$, where $u_i \in A^{\times}, a'_i \in J$ and $m_i \ge 1$.
We can assume that $m_{1} \le m_i$ for all $i \le r'$.

We now let $\alpha_i = u^{-1}_{1}u_i\pi^{m_i -m_1}$ and
$b_i = a_i - \alpha_i a_{1}$ for $2 \le i \le r'$. Then we get
\[
b_i = a'_i + u_i\pi^{m_i} - \alpha_i(a'_{1} + u_{1}\pi^{m_{1}})
= (a'_i - \alpha_i a'_{1}) + (u_i\pi^{m_i} -
u^{-1}_{1}u_i\pi^{m_i -m_1} u_{1}\pi^{m_1})
  = a'_i - \alpha_i a'_1.
\]
In particular, $b_i \in I \cap J$ for $2 \le i \le r'$.
Letting $b_i = a_i$ for $i = 1$ and $i > r'$, it is clear that
$\{b_1, b_2, \ldots , b_r\}$ is a regular sequence in $A$ and
$I = (b_1, b_2, \ldots , b_r)$. This concludes the proof.
\end{proof}

Recall that a Noetherian scheme $X$ is called an $R_i$-scheme if it is regular
in codimension $i$. It is called an $S_i$-scheme if $\sO_{X,x}$ satisfies
Serre's $S_i$ condition (see \cite[Chapter~23, pp.~183]{Matsumura}) for every
$x \in X$. We shall deduce the following results from the Bertini theorems of
\cite[\S~2, 3]{Ghosh-Krishna-Bertini}.

\begin{lem}\label{lem:Bertini-0}
Assume that $k$ is infinite and $X \subset \P^n_k$ is a reduced quasi-projective
scheme of dimension $d \ge 3$. Let $C \subset X$ be a reduced curve.
Let $Z \subset C_\reg \cap X^o$ be a finite closed subset such that
    the intersection of $Z$ with every irreducible component of $C$ which meets
    $X^o$ is nonempty.
Then for all $N \gg 0$, there is a dense open subscheme $\sU$ of the linear system
of hypersurfaces of degree $N$ containing $C$ in $\P^n_k$ such that every
$H \in \sU(k)$ has the property that $H \cap X^o$ is regular away from $C$ and
is also regular at every point of $Z$.
In particular, there is a finite closed subset $Z_H \subset C$ (depending on
$H \in \sU(k)$) such that $(H\cap X^o) \setminus Z_H$ is regular.
\end{lem}
\begin{proof}
Let $\ov{X}$ (resp. $\ov{C}$) be the scheme theoretic closure of $X$ (resp. $C$) in
  $\P^n_k$ so that $\ov{C}$ is a reduced projective curve.
  Let $\sI_{\ov{C}} \subset \sO_{\P^n_k}$ be the sheaf of ideals defining $\ov{C}$.
  By the proof of \cite[Proposition~2.3]{Ghosh-Krishna-Bertini}, it follows that for all
  $N \gg 0$, there is a dense open subscheme $\sU_1$ of the linear system
  $|H^0(\P^n_k, \sI_{\ov{C}}(N))|$ such that every $H \in \sU_1(k)$ has the property
  that $H \cap X^o$ is regular away from $C$.

  We now fix a point $x \in Z$. Let $\fm_x$ be the maximal ideal of $\sO_{X,x}$
  and $S = \Spec(k(x))$.
  Since $X$ and $C$ are both regular at $x$, it follows that
  the canonical restriction map $\sI_{\ov{C}} \to {\sI_S}/{\sI^2_S}$ can not
  be zero. 
  In particular, the canonical map
  $\sI_{\ov{C}}(N) \to {\sI_S}/{\sI^2_S}(N)$ is not zero.
  Since $H^0(\P^n_k, \sI_{\ov{C}}(N))\otimes_k \sO_{\P^n_k} \surj \sI_{\ov{C}}(N)$
  for all $N \gg 0$, 
we conclude that the canonical map of $k$-vector spaces
  \[
    \psi_x \colon H^0(\P^n_k, \sI_{\ov{C}}(N)) \to    {\fm_x}/{\fm^2_x}(N)
  \]
is not zero for all $N \gg 0$.

If we  choose any element $f \in H^0(\P^n_k, \sI_{\ov{C}}(N))$,
then $f$ will not die in ${\fm_x}/{\fm^2_x}(N)$ if and only if
the hypersurface $H$ defined by $f$
contains $C$ and $H \cap X$ is regular at $x$.
Note that this uses our assumption that $x \in X^o$.
Since $\P_k({\rm Ker}(\psi_x))$ is a proper closed subscheme of
$|H^0(\P^n_k, \sI_{\ov{C}}(N))|$, 
it follows that there is a dense open subscheme
$\sU_x \subset |H^0(\P^n_k, \sI_{\ov{C}}(N))|$ for all $N \gg 0$ such that $H \cap X$
is regular at $x$ for every $H \in \sU_x(k)$.
We let $\sU = ({\underset{x \in Z} \cap} \sU_x) \cap \sU_1$.
Then $\sU \subset |H^0(\P^n_k, \sI_{\ov{C}}(N))|$ is open dense and
every $H \in \sU(k)$ has the property that $H \cap X^o$ is regular away from $C$
and is also regular at every point of $Z$.
But such an $H$ must have the
property that $H \cap X^o$ is regular in an open neighborhood $V_H$ of $Z$ in $X^o$.
Letting $Z_H = C \setminus V_H$, we can conclude the proof of the lemma.
\end{proof}

\begin{lem}\label{lem:Bertini}
  Assume that $k$ is infinite and $X \subset \P^n_k$ is an integral quasi-projective
  Cohen-Macaulay scheme of dimension $d \ge 3$. Let $C \subset X$ be a reduced curve
  and let $A \subset X$ be a reduced closed subset.
  Then there is a dense open subscheme $\sU$ of the linear system of hypersurfaces of
  all degrees $N \gg 0$ containing $C$ in $\P^n_k$ such that every
  $H \in \sU(k)$ satisfies the following:
  \begin{enumerate}
  \item
    $H \cap X$ is an integral Cohen-Macaulay scheme of dimension $d-1$.
  \item
    There is a finite closed subset $Z_H \subset C$ (depending on $H$) such that
    $(H\cap X^o) \setminus Z_H$ is regular.
  \item
    $H \cap X$ intersects $A$ properly away from $C$. 
  \end{enumerate}
\end{lem}
\begin{proof}
  By \cite[Theorem~3.4, Lemma~3.7]{Ghosh-Krishna-Bertini},
  there is a dense open subscheme $\sU$ of the linear system of hypersurfaces in
  $\P^n_k$ of all degrees $N \gg 0$  containing $C$ such that every $H \in \sU(k)$
  has the property that
  $Y = H \cap X$ is an irreducible Cohen-Macaulay scheme of dimension $d-1$,
  and it intersects $A$ properly away from $C$.
   In particular, $Y$ satisfies (3).
  We combine this with \lemref{lem:Bertini-0} to see that after possibly
  enlarging $N$ and shrinking $\sU$, the section $Y = H \cap X$ also
  satisfies (2).
In particular, $Y$ is an $R_0$-scheme.
    Since $Y$ is Cohen-Macaulay of positive dimension, it is an $S_1$-scheme.
    But an $(R_0 + S_1)$-scheme is reduced. It follows that $Y$ satisfies (1) too.
  \end{proof}

Let $X$ be an integral quasi-projective $k$-scheme of dimension $d \ge 3$ and
let $D \subset X$ be an effective Cartier divisor on $X$. Let
$C_\pm\subset X$ be two distinct closed subschemes of pure dimension one
whose no component is contained in $D$ and $C_+ \cap D = C_- \cap D$. We let
$\Delta(C) = C_+ \cup C_-$ be the subscheme
defined by the sheaf of ideals $\sI_+ \cap \sI_-$, where $\sI_\pm$ are the sheaves of 
ideals defining $C_\pm$ in $X$. 
 We fix a locally closed embedding $X \inj \P^n_k$ and let $V_m$ be the
 linear system of hypersurfaces in $\P^n_k$ of degree $m$ containing $\Delta(C)$.
 Let $x \in C_\pm \cap D$ be a closed point. We let $R$ and $\fm$ respectively
  denote the local ring $\sO_{X,x}$ and its maximal ideal.

\begin{lem}\label{lem:Open}
   Assume that $X$ is a Cohen-Macaulay scheme.
   Assume further that $C_+$ is a Cartier curve on $X$ at $x$ and $C_-$ is regular at
   $x$. Then for all $m \gg 0$, the scheme $V_m$ is nonempty and there is a dense
   open subscheme $\sU_m \subset V_m$ such that $C_+$ is a Cartier curve on $H \cap X$ at
   $x$ for all $H \in \sU_m(k)$.
 \end{lem}
\begin{proof}
It is easy to see that $V_m \neq \emptyset$ for all $m \gg 0$.
For any coherent sheaf $\sF$ on $X$,
we let $\mu_x(\sF)$ be the minimal number of generators of the stalk
$\sF_x$ as an $\sO_{X,x}$-module. 
For any integer $t \ge 1$, let $\sU_m(t) \subset V_m$ be the subscheme of
points $H \in V_m$ such that $\mu_x({\sI_{C_+}}/{\sI_{H \cap X}}) \le t$, where
$\sI_{Z}$ is the sheaf of ideals of a closed subscheme $Z \subset X$.
Since $X$ is Cohen-Macaulay (in particular, a catenary scheme),
one knows from \cite[Theorem~17.4]{Matsumura} that every $H \in \sU_m(d-2)(k)$
has the property that  $C_+$ is a Cartier curve on $H \cap X$ at $x$.
Hence, it suffices to show that $\sU_m := \sU_m(d-2)$ is a nonempty open subscheme of $V_m$.

It follows from \lemref{lem:Algebra-0} that $\sU_m \neq \emptyset$.
It remains to show that  $\sU_m \subset V_m$ is open.
This part does not require $X$ to be Cohen-Macaulay.
We let $I$ (resp. $I_H$) be the stalk of the ideal sheaf $\sI_{C_+}$ (resp.
  $\sI_{H \cap X}$) inside $X$ at $x$ for any $H \in V_m(k)$.
  For any such $H$, we let $f_H \in R$ be the defining equation of
  $H \cap X$ at $x$ obtained by suitably localizing the homogeneous polynomial
  in $k[x_0, \ldots , x_n]$ that defines $H$.
  We let $J = {I}/{I_H}$ and $\psi_H(1) = f_H$ so that there is an exact sequence of
  $R$-modules:
  \[
    R  \xrightarrow{\psi_H} I \to J \to 0.
  \]
This yields an exact sequence
  \[
    k(x) \xrightarrow{{\psi_H} \otimes 1} I \otimes_R k(x) \to
    J \otimes_R k(x) \to 0.
  \]

By Nakayama's lemma, $\mu_x(J) \le d-2$ if and only if $\mu_x( J \otimes_R k(x))
  \le d-2$. On the other hand, we are given that $\mu_x(I \otimes_R k(x)) = d-1$.
  Hence, the statement $\mu_x( J \otimes_R k(x)) \ge d-1$ is equivalent to the
  statement that $f_H \equiv 0\ {{ \rm mod}}(\fm I)$. But the latter is clearly a closed
  condition on $V_m$. We have thus shown that $\sU_m(d-2)$ is open in $V_m$.
  This concludes the proof.
 \end{proof}

\begin{lem}\label{lem:Cartier-higher}
Assume that $k$ is infinite. 
Let $X \subset \P^n_k$ be an integral quasi-projective $k$-scheme of dimension
$d \ge 3$ and let $D \subset X$ be a Cartier divisor on $X$. Let
$C_\pm \subset X$ be two distinct integral curves not contained in $D$ 
such that $C_+ \cap D = C_- \cap D$. Assume that
$C_\pm$ are Cartier curves on $X$ along $D$ and are regular along $D$.
Assume further that $X$ is a Cohen-Macaulay scheme. Let $\Delta(C) = C_+ \cup C_-$.
Let $A \subset X$ be a reduced closed subset.
Then there exist  independent general hypersurfaces
$H_1, \ldots , H_{d-2} \subset \P^n_k$ of large enough degrees containing $\Delta(C)$
 such that
$Y = X \cap H_1 \cap \cdots \cap H_{d-2}$
satisfies the following:
\begin{enumerate}
\item
  $Y$ is an integral Cohen-Macaulay scheme of dimension $2$.
\item
There is a finite closed subset $Z_Y \subset \Delta(C)$ (depending on $Y$) such that
$(Y\cap X^o) \setminus Z_Y$ is regular.
\item
  $C_\pm$ are Cartier curves on $Y$ along $D$.
  \item
 $Y$ intersects $A$ properly away from $\Delta(C)$. 
\end{enumerate}
\end{lem}
\begin{proof}
This is a direct consequence of Lemmas~\ref{lem:Bertini} and ~\ref{lem:Open}.
\end{proof}

We can now prove a version of \lemref{lem:Cartier} in higher dimensions.

\begin{prop}\label{prop:Cartier-higher-*}
  Assume that $k$ is infinite.
Let $X$ be an integral quasi-projective $k$-scheme of dimension $d \ge 2$. Let
$C_\pm \subset X$ be two distinct integral curves not contained in $D$ 
such that $C_+ \cap D = C_- \cap D$. Assume that
$C_\pm$ are Cartier curves on $X$ along $D$ and are regular along $D$.
Assume further that $X$ is a Cohen-Macaulay scheme.
Then $C_+ \amalg_D C_- \subset S_X$ is a Cartier curve along $D$.
\end{prop}
\begin{proof}
  By \lemref{lem:Cartier}, we can assume $d \ge 3$. In this case, we let
  $\Delta(C) \subset Y \subset X$ be as in \lemref{lem:Cartier-higher}.
  Then we have inclusions $C_+ \amalg_D C_- \subset S_Y \subset S_X$.
  The first inclusion is Cartier along $D$ by \lemref{lem:Cartier} and
  $S_Y$ is a complete intersection in $S_X$ by our choice of $Y$ and
  \propref{prop:double*}. The desired assertion now follows.
\end{proof}

\section{Cartier curves on the double}\label{sec:CC}
Let $k$ be a field. We fix an integral quasi-projective $k$-scheme $X$ of
dimension $d \ge 2$.  Let $D \subset X$
be a closed subscheme such that $D^o = D \cap X^o$ is
a Cartier divisor on $X^o$.
We let $D' = D_\red \cup (X_\sing)_+ \cup (X_\sing)_-  = (S_X)_\sing$.
Let $A, B \subset X$ be reduced closed subschemes of dimensions $\le d-2$  such that
$X_\sing \subset B \subset X_\sing \cup D = D^\star$. We let $\Sigma = S_A \cup S_B =
S_{A\cup B}$.

The most challenging part of the proof of \thmref{thm:Main-1} is the construction
of a suitable pull-back map $\tau^*_X \colon \CH_0(S_X) \to \CH_0(X|D)$.
The strategy of doing this is to first construct a pull-back map
$\tau^*_X \colon \CH^{LW}_0(S_X;A) \to \CH^{LW}_0(X|D; A \cup B)$,
assuming  $k$ is infinite.
We then descend $\tau^*_X$ to the level of $\CH_0(S_X)$ by a projective bundle trick.
The reduction to the case of finite field is done using a pro-$\ell$-extension
trick.

In order to construct $\tau^*_X$ on $\CH^{LW}_0(S_X;A)$, we need to show that
the latter can be described by
relations among the 0-cycles on $(S_X)_\reg$ which are given by the
divisors of rational functions on a very specific kind of Cartier curves on $S_X$.
The goal of this section is to achieve this reduction when $k$ is infinite.
If $S$ is a quasi-projective $k$-scheme and $\sL$ is a line bundle on $S$
with a section $t \in H^0(S, \sL)$, we write $(t)$ for the divisor of zeros of
$t$, that we consider as a closed subscheme of $S$.
We shall assume in this section that $k$ is infinite.
We fix a locally closed embedding $S_X \inj \P^n_k$.
For a subscheme $Z \subset S_X$, we let $Z_\pm = \iota^*_\pm(Z)$.
Recall that $Z^\dagger = (Z_\red)_\reg$ for any $Z \in \Sch_k$.

\begin{lem}\label{lem:Cartier-good}
Assume that $d \ge 3$. Then $\sR^{LW}_0(S_X;S_A)$ is generated by the 
divisors of functions on Cartier curves $C\inj S_X$, where $C$ satisfies the following:
\begin{enumerate}
\item
  There is a reduced surface $Y \subset S_X$ such that $Y_\pm = X_\pm \cap Y$
  are integral.
\item 
There are distinct hypersurfaces \[H_1, \cdots , H_{d-2} \inj \P^n_k\] 
such that 
$Y \cap S_{X^o} = S_{X^o} \cap H_1 \cap \cdots \cap H_{d-2}$ is a complete
intersection.
\item
$Y^o_\pm = X^o_{\pm} \cap Y = X^o_{\pm} \cap H_1 \cap \cdots \cap H_{d-2}$.
\item
No component of $Y$ is contained in $D'$.
\item
$C \subset Y$.
\item
$C$ is a Cartier divisor on $Y$.
\item
  $Y^o_{\pm}$ and $Y^o_\pm \cap D^{\dagger}$ are regular away from $C_\pm$.
\item
  $C$ is Cohen-Macaulay.
\item
  $C \cap \Sigma = \emptyset$.
\item
  $Y \cap \Sigma$ is finite.
\end{enumerate}
\end{lem}    
\begin{proof}
Let $C_1 \inj S_X$ be an arbitrary Cartier curve and let 
$f \in k(C_1, C_1 \cap D')^{\times}$.
Since $C_1$ is Cartier along $D'$, it follows that
it is Cartier in $S_X$ along each of its generic points. 
 We shall now replace $C_1$ by the curves of desired type using
a combination of \cite[Lemma~1.3]{Levine-2} and 
\cite[Lemma~1.4]{Levine-2} as follows.

We are given that $S_X$ is reduced quasi-projective, $X$ is integral and 
$C_1 \inj S_X$ is Cartier along $D'$ which does not meet $S_A$.
Since $S_B \subset (S_X)_\sing$, we can first modify $C_1$ using
\cite[Sublemma~1]{Biswas-Srinivas} (see its proof)
so that we also have $C_1 \cap S_B = \emptyset$.
Since $S_{X^o}$ is Cohen-Macaulay by \propref{prop:double*}, 
we can apply \cite[Lemma~1.3]{Levine-2} and \lemref{lem:Bertini} to find
distinct hypersurfaces $H_1, \cdots , H_{d-2} \subset
\P^n_k$  such that letting $Y = (S_X \cap H_1 \cap \cdots \cap H_{d-2})_\red$, we
have the following:
\begin{enumerate}
\item
  $Y \cap S_{X^o} = S_{X^o} \cap H_1 \cap \cdots \cap H_{d-2}$ is a complete intersection.
\item
  $Y\cap S_{X^o}$ is Cohen-Macaulay.
\item
$X_{\pm} \cap Y = X_{\pm} \cap H_1 \cap \cdots \cap H_{d-2}:= Y_{\pm}$ are
integral.
\item
No component of $Y$ is contained in $D'$.
\item
$C_1 \subset Y$.
\item
$C_1$ is locally principal in $Y$ in a neighborhood of the finite set
$C_1 \cap D'$ and at each generic point of $C_1$.
\item
 $(Y \setminus C_1) \cap X^o_\pm$ and $(Y \setminus C_1) \cap D^\dagger$ are regular.
\end{enumerate}
Since the hypersurfaces are general away from $C_1$
and $\dim(A) \le d-2$,
$S_X \cap H_1 \cap \cdots \cap H_{d-2}$ will have the property that
its intersection with $S_A$ and $S_B$ are finite.

Since $Y \inj S_X$ is constructed as the reduced locus of
a complete intersection of hypersurfaces
of arbitrarily large degrees,
we can, as Levine shows (more precisely, see either the second paragraph of the
proof of \cite[Lemma~1.4]{Levine-2} or part (i) of the proof of
\cite[Sublemma~1]{Biswas-Srinivas}),  
furthermore find a locally principal closed subscheme
$C$ of $Y$ such that the following hold:
\begin{listabc}
\item
$C_1 \subset C$.
\item
$C$ equals $C_1$ at every generic point of $C_1$.
\item
 $\ov{(C \setminus C_1)} \cap C_1 \cap D' = \emptyset$.
\item
  $C \cap S_A = C \cap S_B = \emptyset$.
\item
  $\sO_Y(C)$ is very ample on $Y$.
\end{listabc}

Since $Y$ is reduced, it follows from \cite[Lemma~2.2]{Binda-Krishna} that
it is the join of $Y_\pm$ along $D$ and there is an exact sequence
\begin{equation}\label{eqn:Cartier-good-0}
  0 \to \sO_Y \to \sO_{Y_+} \times \sO_{Y_-} \to \sO_{Y \cap D} \to 0.
\end{equation}

Since $C$ is locally principal in $Y$, it follows that 
$C \cap X_{\pm} = C \cap Y_{\pm}$ are locally principal in $Y_{\pm}$. 
Since $Y_{\pm}$ are integral surfaces, it follows that $C \cap X_{\pm}$ are 
Cartier divisors in $Y_{\pm}$. It follows from ~\eqref{eqn:Cartier-good-0} that 
$C$ is a Cartier divisor in $Y$. Since $Y \cap S_{X^o}$ is Cohen-Macaulay and
$C$ is a Cartier divisor on $Y \cap S_{X^o}$ by (d), it follows that $C$ is
Cohen-Macaulay. It follows from  (c) that $f$ extends to a function 
$g \in k(C, C \cap D')^{\times}$ by setting
$g = f$ on $C_1$ and $g = 1$ on $C \setminus C_1$. In particular, we have
$\divf(f) = \divf(g)$. This concludes the proof.
\end{proof}

We now further refine the rational equivalence by specifying the shape of the Cartier
curves that generate the group of relations.

\begin{lem}\label{lem:reduction-basic}
  Let $X$ be as above with $d \ge 2$. Let $\nu\colon C \inj S_X$ be a connected
  Cartier curve which misses $S_A$. 
Assume that either $d = 2$ or there are inclusions 
$C \subset Y \subset S_X$, where
$Y$ is a reduced, $Y \cap S_{X^o}$ is a complete intersection surface on $S_{X^o}$ and
$C$ is a Cartier divisor on  $Y$, as in \lemref{lem:Cartier-good}.
Let $f \in k(C, C \cap D')^\times$.

We can then find two Cartier curves $\nu'\colon C' \inj S_X$ and 
$\nu''\colon C'' \inj S_X$ satisfying the following:
\begin{enumerate}
\item
There are very ample line bundles $\sL', \sL''$ on $S_X$ and sections
$t' \in H^0(S_X, \sL'), \ t'' \in H^0(S_X, \sL'')$ such that
$C' = Y \cap (t')$ and $C'' = Y \cap (t'')$ (with the convention $Y=S_X$ if $d=2$).
\item
  $C'$ and $C''$ are reduced.
\item
  $C' \cap \Sigma = C'' \cap \Sigma = \emptyset$.
\item
  $C'_\pm$ and $C''_\pm$ are integral curves in $X$,
which are Cartier everywhere on $X$ and are regular along $D$.
\item
  $C'_\pm \cap D^\dagger$ and $C''_\pm \cap D^\dagger$ are regular.
\item
There are functions $f' \in k(C', C' \cap D')^{\times}$ and
$f'' \in k(C'', C'' \cap D')^{\times}$ such that
$\nu'_*({\divf}(f'))  + \nu''_*({\divf}(f''))  = 
\nu_*({\divf}(f))$ in $\sZ_0((S_X)_\reg \setminus S_A)$.  
\end{enumerate}
\end{lem}
\begin{proof}
In this proof, we shall assume that $Y = S_X$ if $d =2$.
In this case, a Cartier curve on $S_X$ along $D' = (S_X)_{\sing}$ must be  
an effective Cartier divisor on $S_X$. 
Hence, we can assume that 
$C$ is an effective Cartier divisor on $Y$ for any $d \ge 2$. 
By \lemref{lem:Cartier-good} for $d\geq 3$ and by using
 \cite[Sublemma~1]{Biswas-Srinivas}  for $d=2$, we 
 can further assume that $C \cap  \Sigma = \emptyset$.
Notice that $Y = Y_+ \amalg_D Y_-$ (see \cite[Lemma~2.2]{Binda-Krishna}).

Since $S_X$ (and hence $Y$) is quasi-projective over $k$ such that
$Y \cap \Sigma$ is finite and $C \cap  \Sigma = \emptyset$, we repeat the argument
for properties (a) $\sim$ (e) in the proof of \lemref{lem:Cartier-good}
to find (after adding a suitable ample divisor to $C$ on $Y$)
an effective Cartier divisor $C'$ on $Y$ such that the following hold:
\begin{enumerate}
\item
$C \subset C'$.
\item
$\ov{C'\setminus C} \cap (C \cap D') = \emptyset$.
\item
  $C' \cap \Sigma = \emptyset$.
  \item
$\sO_{Y}(C')$ is very ample on $Y$.
\end{enumerate}

Setting $g = f$ on $C$ and $g =1$ on $C' \setminus C$, we then see that
$g$ extends to a function in $k(C',C' \cap D')^{\times}$ such that
$\divf(g) = \divf(f)$. We can thus assume that $\sO_{Y}(C)$ is very ample on $Y$.
We choose $t_0 \in H^0(Y, \sO_{Y}(C))$ such that $C = (t_0)$.

Since $\sO_{Y}(C)$ is very ample, and $Y$ is reduced,
we can use \cite[Theorem~3.4]{Ghosh-Krishna-Bertini} to find a new 
section $t_{\infty} \in H^0(Y, \sO_{Y}(C))$, sufficiently general
(see the proof of \cite[Lemma~1.3]{Levine-2}) so that:
\begin{enumerate}
\item
$(t_{\infty})$ is reduced.
\item
$(t_{\infty}) \cap (t_0) \cap D' = \emptyset$.
\item
$(t_{\infty})$ contains no component of $(t_0)$.
\item
  $D'$ contains no component of $(t_{\infty})$.
\item
  $(t_\infty) \cap \Sigma = \emptyset$.
\end{enumerate}

Denote by $C_{\infty}$ the divisor $(t_{\infty})$. Notice that the
function $h = (f,1)$ is meromorphic on $C\cup C_\infty$ and regular
invertible in a neighborhood of $(C \cup C_\infty) \cap D'$.
Since $Y \cap \Sigma$ is a finite closed subset of $S_X$ which
does not meet $C \cup C_\infty$, we can let $h = 1$ on $Y \cap \Sigma$
and this allows us to extend $h$ to a meromorphic function on
$T := (Y \cap \Sigma) \cup C \cup C_\infty$
which is regular invertible in a neighborhood of
$T' := (Y \cap \Sigma) \bigcup ((C \cup C_\infty) \cap D')$.
Let $S$ denote the finite set $\{\text{poles of $f$}\} \cup (C\cap C_{\infty})$.  
Note that $S \cap T' = \emptyset $ and hence $S \subset (S_X)_\reg \setminus S_A$.

Since $S_X$ is reduced quasi-projective, $X_{\pm}$ (as well as $Y_{\pm}$)
are integral closed subschemes of $S_X$ and $S \subset (S_X)_\reg$, we can apply
\cite[Theorem~3.4, Corollary~3.9]{Ghosh-Krishna-Bertini} to find a very ample line
bundle $\sL$ on $S_X$ and a section $s_\infty \in H^0(S_X, \sL)$ 
(see again \cite[Lemma~1.3]{Levine-2}) so that: 
\begin{listabc}
\item $(s_{\infty})$ and $(s_\infty) \cap Y$ are reduced.
\item
$Y_\pm \not\subset (s_\infty)$.
\item
$(s_\infty) \cap X_{\pm}$ and $(s_\infty) \cap Y_{\pm}$ are integral.
\item
$(s_\infty) \cap T' = \emptyset$.
\item
$(s_\infty) \supset S$.
\item
$C \cup C_{\infty}$ contains no component of $(s_{\infty}) \cap Y$.
\end{listabc}

If $\ov{S_X}$ is the closure of $S_X$ in $\P^n_k$ and if $\sI$ is the ideal sheaf of 
$\ov{S_X} \setminus S_X$ in $\P^n_k$, then we can find a 
section $s'_{\infty}$ of the sheaf $\sI \otimes \sO_{\P^n_k}(m)$ for some
$m \gg 0$, which restricts to a section $s_{\infty}$ on $S_X$ satisfying
the properties (a) - (f) on $S_X = \ov{S_X} \setminus V(\sI)$.
This implies in particular that 
$S_X \setminus (s_{\infty}) = \ov{S_X} \setminus (s'_{\infty})$ 
is affine. Set $\sL = \sO_{\P^n_k}(1)$ and $\sL' = \sL^m$.

We now know that $Y^o_+$ and $Y^o_+ \cap D^\dagger$ are
regular away from $C$ and (d) tells us that 
$(s_\infty)$ intersects $Y_+$ along $D$ at only those points which are away 
from $C \cup X_{\sing}$. It follows that $(s_\infty) \cap Y_+ \cap D \subset
(s_\infty) \cap (Y_+)_{\rm reg}$. On the other hand, we can use
\cite[Proposition~2.3]{Ghosh-Krishna-Bertini}
to ensure that $(s_\infty) \cap (Y_+)_{\rm reg}$ is regular.
The same holds for $Y_-$ and $Y_{\pm} \cap  D^\dagger$ as well.
We conclude that we can choose $\sL'$ and $s_\infty \in H^0(S_X, \sL')$ such that
(a) - (f) above as well as the following hold: \\
\hspace*{1cm} g) \ $(s_{\infty}) \cap Y_{\pm}$  is regular along $D$
and $(s_{\infty}) \cap Y_{\pm} \cap D^\dagger$ is regular.
\\
\hspace*{1cm} h) \ $S_X \setminus (s_\infty)$ is affine.

In particular, the function $h$, which is regular outside $S$, extends to
a regular function $H$ on $U = S_X \setminus (s_{\infty})$. 
Since $H$ is a meromorphic function on $S_X$ which has poles only along
$(s_{\infty})$, it follows that $Hs^N_{\infty}$ is an element of 
$H^0(U, (\sL')^N)$ 
which extends to a section $s_0$ of $(\sL')^N$ on all of $S_X$, if we choose
$N \gg 0$. Note that up to replacing $s_0$ by $s_0s^i_\infty$, we are free to choose
$N$ as large as needed.

Since $s_{\infty}$ and $h$ are both invertible on $T$ along 
$T'$, we see that $s_0$ is invertible on $T$ along $T'$.
In particular, $s_0 \notin H^0(S_X, (\sL')^N\otimes \sI_{C \cup C_\infty})$.
Let $V \subset |H^0(S_X, (\sL')^N)|$ be the linear system generated by
$s_0$ and $H^0(S_X, (\sL')^N\otimes \sI_{T})$.
Then for $N \gg 0$, the base locus of $V$ is the finite set of closed points
$(s_0) \cap (C \cup C_\infty)$ and it defines a locally closed embedding of
$S_X \setminus (C \cup C_\infty)$ into $\P_k(V^\vee)$. Recall that
$(s_0) \cap (Y \cap \Sigma) = \emptyset$.

By applying \cite[Lemma~1.3]{Levine-2} and
\cite[Theorems~3.4, 3.8]{Ghosh-Krishna-Bertini},
we can find a section $s'_0 = s_0 + \alpha$, with $\alpha \in 
H^0(S_X, (\sL')^N\otimes \sI_{T})$ such that the following hold:

\begin{listabcprime}
\item
$(s'_0)$ and $(s'_0) \cap Y$ are reduced away from $C \cup C_\infty$.
\item
$Y \not\subset (s'_0)$.
\item
$(s'_0) \cap X_{\pm}$ and $(s'_0) \cap Y_{\pm}$ are integral away from $C \cup C_\infty$.
\item
$(s'_0) \cap T' = \emptyset$.
\item
$C \cup C_{\infty}$ contains no component of $(s'_0) \cap Y$.
\item
  $(s'_0) \cap Y_{\pm}$ are regular along $D$ and $(s'_0) \cap Y_{\pm} \cap D^\dagger$ is
  regular.
\end{listabcprime}

Since $(s'_0) \cap Y$ does not meet $\Sigma$, and $Y$ is Cohen-Macaulay away from
$\Sigma$ (in fact away from $S_B$),
it follows that $(s'_0) \cap Y$ is Cohen-Macaulay. Since a generically
reduced Cohen-Macaulay scheme is reduced (see \cite[Proposition 14.124]{GW}),
it follows from (a') that $(s'_0) \cap Y$ is reduced. By the same token,
$(s'_0) \cap S_{X^o}$ is reduced. Since $(s'_0) \cap Y$ is Cohen-Macaulay,
and hence has pure dimension one, it follows from (c') that $(s'_0) \cap Y_\pm$ are
also irreducible. We conclude that $(s'_0) \cap Y_\pm$ are integral.
By the same token, $(s'_0) \cap X_{\pm}$ are irreducible.
In particular, $(s'_0) \cap X^o_{\pm}$ are integral.

We now have 
\[
\frac{s'_0}{s^N_{\infty}} = \frac{Hs^N_{\infty} + (\alpha s^{-N}_{\infty})
s^N_{\infty}}{s^N_{\infty}} = H + \alpha s^{-N}_{\infty} = H', \ (\mbox{say}).
\]

Since $\alpha$ vanishes along $C \cup C_{\infty}$ and $s_{\infty}$ is
invertible along $U$, it follows that $H'_{|{(C \cup C_{\infty}) \cap U}} =
H_{|{(C \cup C_{\infty}) \cap U}} = h_{|U}$. In other words, we have
${s'_0}/{s^N_{\infty}} = h$ as rational functions on $C \cup C_{\infty}$.
We can now compute:

\[
\nu_*({\rm div}(f)) = (s'_0) \cdot C - N (s_\infty) \cdot C
\]
\[ 
0 = {\rm div}(1) = (s'_0) \cdot C_{\infty} - N (s_\infty)\cdot C_{\infty}.
\]

Setting $(s^Y_{\infty}) = (s_\infty) \cap Y$ and $(s'^Y_0) = (s'_0) \cap Y$, we get
\[
\begin{array}{lll}
\nu_*(\divf(f)) & = & (s'_0) \cdot (C - C_{\infty}) - 
N(s_{\infty})(C - C_{\infty}) \\
& = & (s'^Y_0) \cdot (\divf({t_0}/{t_{\infty}})) - N(s^Y_{\infty}) \cdot
(\divf({t_0}/{t_{\infty}})) \\
& = & \iota_{{s'^Y_0}, *}(\divf(f')) + \iota_{{s^Y_{\infty}}, *}(\divf(f'')),
\end{array}
\]
where $f' = ({t_0}/{t_{\infty}})|_{(s'^Y_0)} \in 
k((s'^Y_0), D' \cap (s'^Y_0))^{\times}$ (by (d')) and 
$f'' = (-N) (t_0/t_{\infty})|_{(s^Y_\infty)} \in 
k((s^Y_\infty), D'\cap (s^Y_\infty))^{\times}$ (by (d)). 
It follows from (g) and (f') that 
$(s'^Y_{0})_{| X_+}, \ (s'^Y_{0})_{| X_-} , \ (s^Y_{\infty})_{| X_+}$ and 
$(s^Y_{\infty})_{| X_-}$ are all regular along $D$ and their intersections with $D^\dagger$ are
regular. All these curves are also integral
and miss $A \cup B$.
Setting $\sL'' = (\sL')^N, \ t' = s'_0$ and $t'' = s_\infty$,
the curves $C' = (t') \cap Y$ and $C'' = (t'') \cap Y$ together with the 
functions $f'$ and $f''$ satisfy the conditions of the Lemma.
\end{proof}

The following application of \lemref{lem:reduction-basic} will be used
in the proof of \propref{prop:Tau>2}.

\begin{lem}\label{lem:PF}
  Let $X$ and other notations be as above. Let $\wt{X}$ be an integral quasi-projective $k$-scheme
and $\wt{D} \subset \wt{X}$ a closed subscheme such that
$\wt{D} \cap \wt{X}_\reg$ is a Cartier divisor on $\wt{X}_\reg$.
Let  $p \colon \wt{X} \to X$ be a proper morphism such that
$p^*(D) \subset \wt{D}$.
Let $\wt{A}$ be a reduced
  closed subscheme of $\wt{X}$ of codimension at least two such that
  $p^{-1}(A \cup B) \subset \wt{A}$.
  Assume that 
   $p$ is a local complete intersection closed
  immersion in an open neighborhood of $p^{-1}(D^o)$. Then $p$ induces a
  push-forward map
  \[
    p_*  \colon \CH^{LW}_0(S_{\wt{X}}; S_{\wt{A}}) \to
    \CH^{LW}_0(S_X;S_A).
  \]
\end{lem}
\begin{proof}
  It is clear from our assumptions that there is a push-forward map
  $p_* \colon \sZ_0((S_{\wt{X}})_\reg \setminus S_{\wt{A}}) \to
  \sZ_0((S_X)_\reg \setminus S_A)$.
  To show that it preserves rational equivalences, it suffices to consider
  Cartier curves $C \subset S_{\wt{X}}$ and functions $f \in k(C, C \cap
  (S_{\wt{X}})_\sing)^{\times}$, where $C \inj S_{\wt{X}}$ satisfies properties (1) - (4) of
  \lemref{lem:reduction-basic}. In this case, it is clear that
  $p_*(\divf(f)) \in \sR^{LW}_0(S_X; S_A)$.
\end{proof}

\section{The pull-back map $\tau^*_X$}\label{sec:PB}
Let $k$ be a field. We fix an integral quasi-projective $k$-scheme $X$ of
dimension $d \ge 1$. Let $D \subset X$
be a closed subscheme such that $D^o = D \cap X^o$ is
a Cartier divisor on $X^o$.
We let $D' = D_\red \cup (X_\sing)_+ \cup (X_\sing)_-  = (S_X)_\sing$.
Let $A, B \subset X$ be reduced closed subschemes of dimensions $\le d-2$  such that
$X_\sing \subset B \subset X_\sing \cup D_\red$. We let $\Sigma = S_A \cup S_B =
S_{A\cup B}$.

We define 
\begin{equation}\label{eqn:tau^*-0}
  \tau^*_X\colon  \sZ_0((S_X)_\reg \setminus S_A) \to
  \sZ_0(X^o \setminus (D \cup A)) \quad \text{by} \quad
\tau^*_X([x]) = \iota^*_+([x]) - \iota^*_-([x]),
\end{equation}
where $x \in (S_X)_\reg$ is a closed point and $\iota_\pm \colon X_\pm \inj
S_X$ are inclusions. We shall call $\tau^*_X$ the `difference homomorphism'.

If $d = 1$, then it is shown in \cite[\S~5.5]{Binda-Krishna} that $\tau^*_X$ descends
to a map
\begin{equation}\label{eqn:tau^*-1}
\tau^*_X\colon  \CH_0^{LW}(S_X) = \CH_0(S_X) \to \CH^{LW}_0(X|D).
\end{equation}

The goal of this section is to generalize this to higher dimensions when
$k$ is infinite. We shall therefore assume throughout this section that
$k$ is infinite.

\subsection{The case of surfaces}\label{sec:dim-2}
We consider the case $d = 2$ in this subsection.
We shall prove the higher dimension case by reduction to this case.

\begin{lem}\label{lem:Tau-2}
 Assume that $d =2$.  
Then $\tau^*_X$ descends to a homomorphism
  \[
\tau^*_X\colon  \CH_0^{LW}(S_X;S_A)  \to \CH^{LW}_0(X|D; A \cup B).
  \]
\end{lem}
\begin{proof}
We fix a locally closed embedding $S_X \inj \P^n_k$.  
We have shown in \lemref{lem:reduction-basic} that in order to
prove that $\tau^*_X$ preserves the subgroups of rational equivalences,
it suffices to show that $\tau^*_X(\divf(f)) \in \sR^{LW}_0(X|D;A \cup B)$,
where $f$ is a 
rational function (which is regular and invertible along $D'$)
on a Cartier curve $\nu\colon  C \inj S_X$ that we can 
choose in the following way:
\begin{enumerate}
\item
There is a very ample line bundle $\sL$ on $S_X$ and a section
$t \in H^0(S_X, \sL)$ such that $C = (t)$.
\item
  $C$ is a reduced Cartier divisor of the form $C = C_+ \amalg_E C_-$,
  where $E = \nu^*(D)$.
\item
  $C \cap \Sigma = \emptyset$.
\item
The restrictions of $C$ to $X$ via the two closed 
immersions $\iota_{\pm}$ are integral curves $C_\pm$ in $X$,
which are Cartier everywhere, are regular along $D$ and whose intersections with
$D^\dagger$ are regular.
\end{enumerate}

If $E = \emptyset$, then $C_\pm$ are two integral curves on $X$
away from $D \cup A\cup B$ and $\tau^*_X(\divf(f)) = \divf(f|_{C_+}) - \divf(f|_{C_-})
\in \sR^{LW}_0(X|D;A \cup B)$. We can thus assume that $E \neq \emptyset$.
Then (1) implies that 
\begin{equation}\label{eqn:map-diff-surfaces-0*}
(t_+)_{|D} = \iota'^*(t) = (t_-)_{|D},
\end{equation}
where recall that $\iota' = \iota_+ \circ \iota = \iota_- \circ \iota \colon  
D \inj S_X$ denotes the inclusion map. 

Let $(f_+, f_-)$ be the image of $f$ in $\sO^{\times}_{C_+, E} \times 
\sO^{\times}_{C_-,E} \inj k(C_+) \times k(C_-)$.
It follows from \cite[Lemma~2.2]{Binda-Krishna} that there is an exact sequence
\begin{equation}\label{eqn:D*}
0 \to \sO_{C,E} \to \sO_{C_+,E} \times \sO_{C_-, E} \to \sO_E \to 0.
\end{equation}
In particular, we have 
\begin{equation}\label{eqn:map-diff-surfaces-0} 
(f_+)_{|E} = (f_-)_{|E} \in \sO^{\times}_E .
\end{equation}

Let us first assume that $C_+ = C_-$ as curves on $X$. 
Let $C'$ denote this curve. 
Since $f$ is regular and invertible on $C$ along $D'$, we get 
$f_+, f_- \in \sO^{\times}_{C',E}$.
Setting $g := f_+ f^{-1}_- \in \sO^{\times}_{C',E}$, it
follows from ~\eqref{eqn:map-diff-surfaces-0} that
$g \in \Ker(\sO^{\times}_{C',E}  \to \sO^{\times}_E)$. 
Moreover, $\tau^*_X(\divf(f)) = \iota^*_+(\divf(f)) -
\iota^*_-(\divf(f)) = \divf(f_+) - \divf(f_-) = \pi_*(\divf(g))$,
where $\pi \colon C' \inj X$ is the inclusion.
We conclude from (2) and (4) that
$\tau^*_X(\divf(f))$ dies in $\CH^{LW}_0(X|D;A \cup B)$.

We now assume that $C_+ \neq C_-$ and write  $\Delta(C) = C_+ \cup C_- \subset X$.
If $\sI_{C_\pm}$ are the ideal sheaves defining $C_\pm$, then $\Delta(C)$ is defined by
$\sI_{C_+} \cap \sI_{C_-}$. Since $\sI_{C_\pm}$ are sheaves of prime ideals, it
follows that $\Delta(C)$ is reduced with two irreducible components $C_\pm$.
Furthermore, $\Delta(C) \cap (A \cup B) = \emptyset$.
Since $C_+ \cap D = C_- \cap D = E$ as closed subschemes, 
we see that the support of $\Delta(C) \cap D$ is same as $E_\red$.

Since $A \cup B$ is a finite closed subscheme of $X$
and $\Delta(C) \cap (A \cup B) = \emptyset$, the functions 
$f_\pm$ extend to meromorphic functions on
$T_\pm = A \cup B \cup C_\pm$ which are regular invertible in a neighborhood of 
$T' = A \cup B \cup E$ by letting $f_\pm = 1$ on $A \cup B$.
Let $S_\pm$ denote the set of closed points
on $C_\pm$, where $f_\pm$ have poles.
We let $T = T_+ \cup T_- = A \cup B \cup \Delta(C)$ and $S = S_+ \cup S_-$.
It is clear that $S \cap (D \cup A \cup B) = \emptyset$. In particular, $S \subset X^o$.
We now repeat the constructions in the proof of \lemref{lem:reduction-basic}
to find a very ample line bundle $\sL$ on $X$ and a section 
$s_{\infty} \in H^0(X, \sL)$ (see \cite[Lemma~1.4]{Levine-2} and
\cite[Corollary~3.9]{Ghosh-Krishna-Bertini}) such that:
\begin{listabc}
\item $(s_{\infty})$ is integral (because $X$ is integral);
\item $(s_\infty) \cap T' = \emptyset$;
\item $(s_\infty) \supset S$;
\item $(s_{\infty}) \not\subset \Delta(C)$;
\item $(s_{\infty})$ is regular away from $S$;
\item $(s_\infty) \cap D^\dagger$ is regular.
\item $X \setminus (s_\infty)$ is affine.
\end{listabc}

It follows that $f_{\pm}|_{T_\pm \setminus (s_\infty)}$ 
extend to regular functions $F_{\pm}$ on
$X \setminus (s_\infty)$. Since $F_{\pm}$ are meromorphic functions
on $X$ which have poles only along $(s_\infty)$, it follows that 
$F_{\pm}s^N_{\infty}$ are elements of $H^0(X \setminus (s_\infty), \sL^N)$ which 
extend to sections $s_0^{\pm}$ of $\sL^N$ on all of $X$, if we choose $N \gg 0$.
Since the functions $s_\infty$ and $F_\pm$ are all meromorphic functions
on $X$ which are regular on $X \setminus (s_\infty)$, 
it follows that each of them 
restricts to a meromorphic function on $T$ which is regular on
$T \setminus (s_\infty)$ and 
$F_\pm|_{T\pm \setminus (s_\infty)} = f_\pm|_{T_\pm \setminus (s_\infty)}$.
Since $s_\infty$ and $F_\pm$ are invertible on $T_\pm$ in some neighborhood of
$T'$, we see that $s_0^{\pm}$ are invertible on $T_\pm$ in some
neighborhood of $T'$.
In particular, $s_0^{\pm} \notin H^0(X, \sL^N\otimes \sI_{T})$.
An argument identical to the one in the proof of \lemref{lem:reduction-basic}
allows us to find sections ${s'_0}^{\pm} = s_0^{\pm} + \alpha^\pm$ with
$\alpha^\pm \in H^0(X, \sL^N\otimes \sI_{T})$ such that $({s'_0}^{\pm})$ satisfy the
following:
\begin{listabcprime}
 \item $({s'_0}^{\pm})$ are integral;
\item
$({s'_0}^{\pm}) \not\subset \Delta(C)$; 
\item $({s'_0}^{\pm}) \cap T'  = \emptyset$;
\item
  $({s'_0}^{\pm})$ are regular  along $D$.
\item
  $({s'_0}^{\pm}) \cap D^\dagger$ are regular.
\end{listabcprime} 

We then have 
\begin{equation}\label{eqn::map-diff-surfaces-10}
\frac{{s'_0}^{\pm}}{s^N_{\infty}} = 
\frac{F_{\pm}s^N_{\infty} + (\alpha^{\pm} s^{-N}_{\infty})
s^N_{\infty}}{s^N_{\infty}} = F_{\pm} + \alpha^{\pm} s^{-N}_{\infty} = H_{\pm}, \ 
(\mbox{say}).
\end{equation}

We can now find a dense open subscheme $U' \subset 
X \setminus ((s_\infty) \cup ({s'_0}^+) \cup ({s'_0}^-))$
which contains $T'$ and where $F_\pm, \alpha^\pm, s_\infty$ and 
${s'_0}^\pm$ are all regular. 
In particular, $H_\pm$ are rational functions on $X$
which are regular on $U'$.

Since $T' \subset U'$ and $E \neq \emptyset$, 
it follows that $T_\pm \cap U'$ are dense open in $T_\pm$.
It follows that $s_\infty$ and ${s'_0}^\pm$ restrict to regular functions
on $T_\pm \cap U'$ which are invertible in a neighborhood of $T'$.
Since $\alpha^{\pm}$ vanish along $T$ and $s_{\infty}$ is
invertible on $U'$, it follows that ${H_{\pm}}_{|{T_\pm \cap U'}} =
{F_{\pm}}_{|{T_\pm\cap U'}}$.
Since $F_\pm$ restrict to regular functions on $T_\pm \cap U'$ which
are invertible along $T'$, it follows that $H_\pm$ restrict to
regular functions on $T_\pm \cap U'$ which are invertible along $T'$.

As $H_+$ (resp. $H_-$) is regular on $U'$, it restricts to
a regular function on the dense open subset $C_- \cap U'$ 
(resp. $C_+ \cap U'$) of
$C_-$ (resp. $C_+$). Furthermore,  we have
\begin{equation}\label{eqn:map-diff-surfaces-1*}
{H_+}_{|E} = {F_+}_{|E} = {f_+}_{|E} \ =^{\dagger} \ {f_-}_{|E} 
= {F_-}_{|E} = {H_-}_{|E},
\end{equation}
where $\dagger$ follows from ~\eqref{eqn:map-diff-surfaces-0}.

We thus saw above that $H_+$ and $H_-$ are both regular functions on
$C_- \cap U'$ such that $H_- \neq 0$. 
In particular, ${H_+}/{H_-}$ is a rational function on $C_-$. 
Since ${s'_0}^+$ and ${s'_0}^-$ are both regular
 functions on $C_- \cap U'$ which are invertible in a neighborhood of $C_-\cap D$, 
it follows that the restriction of ${{s'_0}^+}/{{s'_0}^-}$ on $C_-$ is a 
rational function on $C_-$, which is regular and invertible  
in a neighborhood of   $C_- \cap D$.
On the other hand, we have 
\begin{equation}\label{eqn:map-diff-surfaces-2}
\frac{{s'_0}^+}{{s'_0}^-} = \frac{{{s'_0}^+}/{s^N_\infty}}{{{s'_0}^-}/{s^N_\infty}}
= \frac{H_+}{H_-}, 
\end{equation}
as rational functions on $X$.
In particular, ${s'_0}^+ \cdot {H_-} = {{s'_0}^-} \cdot {H_+}$
as regular functions on $U'$.
Hence, this identity holds after restricting these regular functions
to $C_- \cap U'$. We thus get 
\begin{equation}\label{eqn:map-diff-surfaces-20}
\frac{{s'_0}^+}{{s'_0}^-} = \frac{{{s'_0}^+}/{s^N_\infty}}
{{{s'_0}^-}/{s^N_\infty}} = \frac{H_+}{H_-}, 
\end{equation}
as rational functions on $C_-$. Note that $H_-$ is nonzero on $C_-$.
Since $\frac{{s'_0}^+}{{s'_0}^-}$ restricts to a rational function on $C_-$
which is regular and invertible in a neighborhood of   $C_- \cap D$,
 we conclude
that ${H_+}/{H_-}$ restricts to an identical rational function on $C_-$, 
which is  regular and invertible  in a neighborhood of   $C_- \cap D$. 

We now compute
\[
\begin{array}{lll}
\tau^*_X(\divf(f)) & = & \iota^*_+(\divf(f)) -
\iota^*_-(\divf(f)) \\
& = & 
\divf(f_+) - \divf(f_-) \\
& = & 
\left[({s'_0}^+) \cdot C_+ - (s^N_{\infty}) \cdot C_+\right]
- \left[({s'_0}^-)  \cdot C_- - (s^N_{\infty}) \cdot C_-\right] \\
& = & \left[({s'_0}^+)  \cdot C_+ - ({s'_0}^+)  \cdot C_-\right]
+ \left[({s'_0}^+)  \cdot C_- - ({s'_0}^-)  \cdot C_-\right] \\
& & - \left[(s^N_{\infty}) \cdot C_+ - (s^N_{\infty}) \cdot C_-\right] \\
& = & \left[({s'_0}^+)  \cdot (C_+ - C_-)\right] +
\left[C_- \cdot (({s'_0}^+)  - ({s'_0}^-) )\right] -
\left[(s^N_{\infty}) \cdot (C_+ - C_-)\right] \\
& = & ({s'_0}^+)  \cdot (\divf({t_+}/{t_-})) + C_- \cdot 
(\divf({{s'_0}^+}/{{s'_0}^-})) - N(s_{\infty}) \cdot (\divf({t_+}/{t_-})) \\
& = &  ({s'_0}^+)  \cdot (\divf({t_+}/{t_-})) + C_- \cdot 
(\divf({H_+}/{H_-})) - N(s_{\infty}) \cdot (\divf({t_+}/{t_-})).
\end{array}
\]

It follows from (b) and (c') that $t_{\pm}$ restrict to regular invertible
functions on $({s'_0}^+)$ and $(s_{\infty})$ along $D$.
We set $h_1 = {(\frac{t_+}{t_-})}_{| ({s'_0}^+)}, \ 
h_2 = {(\frac{H_+}{H_-})}_{| C_-}$ and
$h_3 =  {(\frac{t_+}{t_-})}_{| s_{\infty}}$.
Let $\nu_1\colon ({s'_0}^+) \inj X$, $\nu_2\colon  C_-  \inj X$ and
$\nu_3\colon  (s_\infty) \inj X$ denote the inclusion maps.
We now note that $({s'_0}^+), \ C_-$ and $(s_\infty)$ are integral Cartier divisors on
$X$ which are regular along
$D$ by (4), (e) and (d'). Their intersections with $D^\dagger$ are regular by (4), (f) and
$(e')$. Furthermore, none of these meets $A \cup B$
by (3), (b) and (c'). Combining these with
~\eqref{eqn:map-diff-surfaces-0*} and ~\eqref{eqn:map-diff-surfaces-1*},
we conclude that the cycles $(\nu_1)_*(h_1), \ (\nu_2)_*(h_2)$ and $(\nu_3)_*(h_3)$
die in $\CH^{LW}_0(X|D; A \cup B)$. This finishes the proof.
\end{proof}

\subsection{Case of higher dimensions}\label{sec:d>2}
We now extend \lemref{lem:Tau-2} to higher dimensions. 
We let $X$ be as described in the beginning of \S~\ref{sec:PB}. 

\begin{prop}\label{prop:Tau>2}
  Assume that $d \ge 2$. Then $\tau^*_X$ descends to a homomorphism
  \[
\tau^*_X\colon  \CH_0^{LW}(S_X;S_A)  \to \CH^{LW}_0(X|D;A \cup B).
  \]
\end{prop}
\begin{proof}
  We can assume $d \ge 3$ by \lemref{lem:Tau-2}.
We fix a locally closed embedding $S_X \inj \P^n_k$.
Let $\nu\colon  C \inj S_X$ be a connected Cartier curve which misses $S_A$ and let 
$f \in \sO^{\times}_{C, E}$, where $E = \nu^*(D)$.
We need to show that $\tau^*_X(\divf(f))$ dies in $\CH^{LW}_0(X|D;A \cup B)$.

We can assume that $C$ is a reduced Cartier curve
with an inclusion $C \inj Y$, where $Y$ is as given by
Lemma~\ref{lem:Cartier-good} and $C$ satisfies the properties (1) - (5) of
\lemref{lem:reduction-basic}.
We shall follow the notations of \lemref{lem:reduction-basic}.

We write $C = (t) \cap Y$, where $t \in H^0(S_X, \sL)$ such that
$\sL$ is a very ample line bundle on $S_X$.
Let $t_{\pm} = \iota^*_{\pm}(t) \in H^0(X, \iota^*_{\pm}(\sL))$ and let
$C_{\pm} = (t_{\pm}) \cap Y = (t_{\pm}) \cap Y_{\pm}$.
Let $\nu_{\pm}\colon  C_\pm \inj X$ denote the inclusions.
It follows from our choice of the section that $C_\pm \subset {X^o}$ are
integral curves, which are Cartier everywhere on $X$ and regular along $D$.
Moreover, $C_\pm \cap (A \cup B) = \emptyset$ and $C_\pm \cap D^\dagger$
are regular.

If $C_+ = C_-$, exactly the same argument as in the
case of surfaces applies to show that 
$\tau^*_X(\divf(f))$ dies in $\CH^{LW}_0(X|D;A \cup B)$.
So we assume $C_+ \neq C_-$. We can also assume $E \neq \emptyset$.

Let $\Delta(C) = C_+ \cup C_-$ denote the scheme theoretic image
in $X$ under the finite map $\Delta$. Then $\Delta(C) \cap (A \cup B) =
\emptyset$.
By applying \lemref{lem:Cartier-higher} (to $X^o$ instead of $X$), we find that
there are general hypersurfaces
$H_1, \ldots , H_{d-2} \subset \P^n_k$ of large enough degrees such that
$T = (X \cap H_1 \cap \cdots \cap H_{d-2})_\red$ has the following 
properties:
\begin{enumerate}
\item
  $T$ is integral and $T \nsubset D$. 
\item
  $T^o = X^o \cap H_1 \cap \cdots \cap H_{d-2}$.
\item
  $\Delta(C) \subset T$.
\item
  $C_\pm$ are Cartier divisors on $T$ along $D$.
\item
  There is a finite closed subset $Z_T \subset \Delta(C)$ (depending on $T$) such that
  $(T \cap X^o) \setminus Z_T$ is regular.
\item
  $T \cap (A \cup B)$ is finite.
\end{enumerate}

It follows from the above properties that $T_\sing \subset (X_\sing \cap T) \cup Z_T$.
In particular, $T_\sing$ is finite. We write $Z = Z_T \cap T_\sing = Z_1 \amalg Z_2$,
where $Z_1 \subset D$ and $Z_2 \cap D = \emptyset$.
Since $Z_2$ is a finite set of closed points, we can find a finite sequence of
blow-ups with centers lying over $Z_2$ such that the composition 
$\pi \colon \wt{T} \to T$ has the properties: 
\begin{enumerate}
\item
  $\pi$ is isomorphism over $T \setminus Z_2$;
\item
  $\wt{T}$ is an integral surface with
  $\wt{T}_\sing \subset \pi^{-1}(T_\sing \setminus Z_2) \cong T_\sing \setminus Z_2
  \subset (B \cap T) \cup Z_1$.
 \end{enumerate}

We let $p \colon \wt{T} \to X$ be the composite map.
We let $\wt{A} = \pi^{-1}((A \cup B) \cap T) \cong
(A \cup B) \cap T, \ \wt{B} = \wt{T}_\sing$ 
 and $\wt{D} = p^{*}(D)$. 
Then $\wt{A}$ and $\wt{B}$ are finite.
In particular, $\wt{T}$ satisfies the hypotheses of \lemref{lem:Tau-2}.
Let $S_{\wt{T}}$ be the double of $\wt{T}$ along $\wt{D}$.
Since $T^o \inj X^o$ is a complete intersection and
$\pi$ is isomorphism away from $\pi^{-1}(Z_2)$, it follows that $p$ is a
proper morphism which is
a complete intersection closed immersion away from $p^{-1}(X_\sing \cup Z_2)$
(in particular, over an open neighborhood of $D^o$).
Since $X_\sing \subset B$, 
it follows that $p((S_{\wt{T}})_\reg \setminus S_{\wt{A}}) \subset (S_X)_\reg \setminus S_A$.
We can therefore apply \lemref{lem:PF} to get a push-forward map
\begin{equation}\label{eqn:Tau>2-0}
  p_* \colon \CH^{LW}_0(S_{\wt{T}};S_{\wt{A}}) \to \CH^{LW}_0(S_X;S_{A}).
\end{equation}

One checks from the above that there is also a push-forward map
\begin{equation}\label{eqn:Tau>2-1}
  p_* \colon \CH^{LW}_0(\wt{T}|\wt{D}; \wt{A} \cup \wt{T}_\sing) \to
      \CH^{LW}_0(X|D;A \cup B).
\end{equation}

We let $\wt{C}_\pm$ be the strict transforms of $C_\pm$ under $\pi$.
      Since the latter is an isomorphism over $D \cap T$, it follows that
      $\wt{C}_+ \cap \wt{D} = \wt{C}_- \cap \wt{D} = E$ and
      $\wt{C} = \wt{C}_+ \amalg_{E} \wt{C}_-$ is a Cartier curve on $S_{\wt{T}}$
      relative to the singular locus of $S_{\wt{T}}$ (see \lemref{lem:Cartier}).
      Furthermore, $f \in k(\wt{C}, E)^{\times}$ and
      $\divf(f)_C = \pi_*(\divf(f)_{\wt{C}})$ (see \cite[Proposition~1.4]{Fulton}).
It follows by \lemref{lem:Tau-2} that $\tau^*_{\wt{T}}(\divf(f)_{\wt{C}}) = 0$ in 
      $\CH^{LW}_0(\wt{T}|\wt{D}; \wt{A} \cup \wt{T}_\sing)$.
 We thus get
      \[
        \tau^*_X(\divf(f)_C) = \tau^*_X \circ p_*(\divf(f)_{\wt{C}}) =
        p_* \circ \tau^*_{\wt{T}}(\divf(f)_{\wt{C}}) = 0
      \]
      in $\CH^{LW}_0(X|D;A \cup B)$. This concludes the proof.
    \end{proof}

   \subsection{A moving lemma for Chow groups with
      modulus}\label{sec:ML}
We shall now give an application of the existence of $\tau^*_X$.
Let $X$ and other notations be as described in the beginning of \S~\ref{sec:PB}.
Recall our notation $Z^\star = Z \cup X_\sing$ for any $Z \subset X$.
We let $p^+_* \colon \sZ_0(X^o \setminus (D \cup A)) \to 
\sZ_0(S_X \setminus (D' \cup S_A))$ be the push-forward map induced by the
inclusion $\iota_+ \colon X_+ \inj S_X$. One can define $p^-_*$ in a
similar way. Following the proof of \cite[Proposition~5.9]{Binda-Krishna}, one 
can deduce that 
$p^\pm_*$ descend to the level of Chow groups and give rise to the push-forward maps
\begin{equation}\label{eqn:PF-0}
  p^\pm_* \colon \CH^{LW}_0(X|D;A^\star) \to \CH^{LW}_0(S_X;S_A).
\end{equation}
Composing with the canonical surjection
$\CH^{LW}_0(X|D; A\cup B) \surj \CH^{LW}_0(X|D;A^\star)$, we get the push-forward maps

\[
  \wt{p}^\pm_* \colon \CH^{LW}_0(X|D; A\cup B) \to \CH^{LW}_0(S_X;S_A).
\]
It is clear that the composite maps
\begin{equation}\label{eqn:PFF-0}
 \CH^{LW}_0(X|D; A\cup B) \xrightarrow{\wt{p}^\pm_*} \CH^{LW}_0(S_X;S_A)
 \xrightarrow{\tau^*_X}  \CH^{LW}_0(X|D; A\cup B)
\end{equation}
are identity.
As an immediate consequence of this, we get the following moving lemma for the
Levine-Weibel Chow group with modulus.

\begin{cor}\label{cor:ML-0}
  The canonical map
  \[
    \CH^{LW}_0(X|D; A\cup B) \to  \CH^{LW}_0(X|D;A^\star)
  \]
  is an isomorphism.
  \end{cor}

\subsection{The maps $p^\pm_*$ and $\tau^*_X$ for lci
  Chow groups}\label{sec:lci}
We shall now define the maps $p^\pm_*$ and $\tau^*_X$ between the lci Chow
groups with and without modulus $\CH_0(S_X;S_A)$ and $\CH_0(X|D;A \cup B)$ over
any base field.

\begin{prop}\label{prop:lci-maps}
  Let $k$ be any field. Then the maps $p^\pm_*$ and $\tau^*_X$ between the
  0-cycle groups induce the maps
  \[
    p^\pm_* \colon \CH_0(X|D;A^\star) \to \CH_0(S_X;S_A) \ \ \mbox{and} \ \
    \tau^*_X \colon \CH_0(S_X;S_A) \to \CH^{\rm mod}_0(X|D;A \cup B)
  \]
  such that the composite maps
  \[
    \CH^{\rm mod}_0(X|D;A \cup B) \surj \CH_0(X|D;A^\star) \xrightarrow{p^\pm_*}
    \CH_0(S_X;S_A) \xrightarrow{\tau^*_X} \CH^{\rm mod}_0(X|D;A \cup B)
  \]
  are identity.
\end{prop}
\begin{proof}
  The existence of $p^\pm_*$ has the same proof as that of
  \cite[Proposition~5.9]{Binda-Krishna}. If $k$ is infinite, then
  the assertion that $\tau^*_X$ descends to the level of Chow groups,
 is a direct consequence of \propref{prop:Tau>2} and the projective bundle trick of
  \cite[Lemma~8.3]{Gupta-Krishna-BF}, combined with the following two remarks.
  The first is that the cited result is proven when $X$ is regular, but its proof
  only uses that $X$ has singularity in codimension $\ge 2$.
  The second is that for any smooth proper morphism
  $\pi \colon X' \to X$, the push-forward map between the 0-cycle
  groups $\sZ_0(X'^o \setminus \pi^{-1}(D^o \cup A \cup B)) \to \sZ_0(X^o \setminus
  (D^o \cup A \cup B))$ induces the map
  $\pi_* \colon \CH^{\rm mod}_0(X'|D'; \pi^{-1}(A \cup B)) \to
  \CH^{\rm mod}_0(X|D;A \cup B)$, where $D' = \pi^*(D)$.
  
  When $k$ is finite, the proof given in \cite[Theorem~7.1]{Binda-Krishna} when
  $X$ is regular works verbatim (using the above push-forward maps) to
reduce to the case when $k$ is infinite. The rest of the lemma is immediate.
\end{proof}

\begin{remk}\label{remk:Surf-**}
  If $k$ is infinite and $X$ is a surface, then the map
  $\CH^{LW}_0(S_X; S_A) \to \CH_0(S_X; S_A)$ is an isomorphism
  by \cite[Theorem~8.1]{BKS}. It follows therefore from ~\eqref{eqn:PFF-0} and
  \propref{prop:lci-maps} that the map
  $\phi_{X|D} \colon
  \CH^{LW}_0(X|D; A) \to \CH_0(X|D; A)$ is an isomorphism.  This proves the
  promised isomorphism ~\eqref{eqn:Mod-old}.
\end{remk}

\vskip .3cm

As an immediate consequence of \propref{prop:lci-maps}, we get the following 
moving lemma for the Chow group with modulus over arbitrary fields.
This is new even when $X$ is regular.

\begin{cor}\label{cor:mL-lci}
  Let $k$ be any field.
 Then the  canonical map
  \[
    \CH^{\rm mod}_0(X|D; A\cup B) \to  \CH_0(X|D;A^\star)
  \]
  is an isomorphism.
\end{cor}

\vskip .3cm

\subsection{Proof of \corref{cor:ML-CH-mod}}
The first part of the corollary follows directly from 
\corref{cor:mL-lci}. 
If we assume further that $k$ is algebraically closed, then
we have the commutative diagram
\begin{equation}\label{eqn:ML-CH-mod-0}
  \xymatrix@C.8pc{
    \CH^{LW}_0(X|D;B) \ar[r] \ar@{->>}[d] & \CH^{LW}_0(S_X) \ar@{->>}[d] \\
    \CH^{\rm mod}_0(X|D;B) \ar[r] & \CH_0(S_X),}
  \end{equation}
  in which the vertical arrows are surjective. We have shown above that
  the horizontal arrows are injective. The right vertical arrow is an
  isomorphism by \cite[Theorem~1.3]{Gupta-Krishna-JAG} (if $X$ is affine) and
  \cite[Theorem~6.6]{Krishna-torsion} (if $X$ is projective). It follows that
  the left vertical arrow is also an isomorphism, which implies our
  assertion.
 $\hfill\square$

\vskip .3cm

\section{The exact sequence and applications}\label{sec:ES}
The goal of this section is to prove \thmref{thm:Main-1} and its applications
to Bloch's formula and moving lemma for singular surfaces.
We let $k$ be a field and $X$ an integral quasi-projective scheme of dimension
$d \ge 1$ which is regular in codimension one. 
Let $D \subset X$ be a closed subscheme such that
$D^o = D \cap X^o$ is a Cartier divisor on $X^o$. We let $A \subset X$ be a
reduced closed subscheme such that $\dim(A) \le d-2$.

\subsection{The pull-back maps $\iota^*_\pm$}\label{sec:PB-1}
We begin by defining the maps $\iota^*_\pm \colon
\CH_0(S_X;S_A) \to \CH_0(X, D;A)$. At the level of the cycle groups,
$\iota^*_\pm \colon \sZ_0((S_X)_\reg \setminus S_A) \to
\sZ_0(X^o \setminus (D \cup A))$ are simply the projection maps
\begin{equation}\label{eqn:Proj}
  \sZ_0((S_X)_\reg \setminus S_A) = \sZ_0(X^o_+ \setminus (D \cup A))
  \oplus \sZ_0(X^o_- \setminus (D \cup A)) \to
  \sZ_0(X^o \setminus (D \cup A))
\end{equation}
onto the two factors of the direct sum.

\begin{lem}\label{lem:Proj-0}
  The maps $\iota^*_\pm$ induce homomorphisms
  \[
\iota^*_\pm \colon \CH_0(S_X;S_A) \to \CH_0(X, D;A).
\]
\end{lem}
\begin{proof}
 By repeating the argument for the construction of $\tau^*_X$ in
  \propref{prop:lci-maps}, using easy variants of \cite[Proposition~6.1]{Binda-Krishna}
and   \cite[Lemma~8.3]{Gupta-Krishna-BF},
  we at once reduce to showing that $\iota^*_\pm$ induce homomorphisms
\begin{equation}\label{eqn:Proj-1}
 \iota^*_\pm \colon \CH^{LW}_0(S_X;S_A) \to \CH^{LW}_0(X, D;A)
\end{equation}
when $k$ is infinite. We show this for $\iota^*_+$ as the other case is identical.

We let $D' = D_\red \cup (X_\sing)_+ \cup (X_\sing)_- \subset S_X$.
Let $C \subset S_X$ be a Cartier curve missing $S_{A}$
and let $f \in k(C, C \cap D')^{\times}$.
By Lemmas~\ref{lem:Cartier-good} and ~\ref{lem:reduction-basic}, we can assume that
there is a reduced surface $Y \subset S_X$ containing $C$ such
that the following hold:
\begin{enumerate}
\item
  $Y \cap S_{X^o}$ is a complete intersection in $S_{X^o}$.
\item
  $Y_\pm$ are integral surfaces in $X_\pm$ not contained in $D$.
\item
  $C \subset Y$ is a reduced Cartier divisor.
  \item
    $C \cap S_{A^\star} = \emptyset$.
  \item
    $C \cap X_\pm = C_\pm$ are integral curves which are Cartier everywhere on $X_\pm$
    and are regular along $D$.
    \end{enumerate}

If we let $f = (f_+, f_-) \in k(C, C \cap D')^{\times}$, then
    $\iota^*_+(\divf(f)) = \divf(f_+)$.
    On the other hand, $C_+$ is Cartier along $D$ and misses $A^\star$
    by (4) and (5). Moreover, $f_+ \in k(C_+, C_+ \cap D^\star)^{\times}$ since
    $f \in k(C, C \cap D')^{\times}$. It follows that $\iota^*_+(\divf(f))$
    dies in $\CH^{LW}_0(X, D;A)$.
    This concludes the proof.
\end{proof}

\subsection{The map $\Delta^*_X$}\label{sec:Delta*}
We shall now construct $\Delta^*_X \colon \CH_0(X, D;A) \to
  \CH_0(S_X;S_A)$ which will be a right inverse of $\iota^*_\pm$.

  At the level of the cycle groups,
  $\Delta^*_X$ is the diagonal map
  \begin{equation}\label{eqn:Delta-0}
  \sZ_0(X^o \setminus (D \cup A)) \xrightarrow{\Delta^*_X}
  \sZ_0(X^o_+ \setminus (D \cup A)) \oplus \sZ_0(X^o_- \setminus (D \cup A)) =
  \sZ_0((S_X)_\reg \setminus S_A).
  \end{equation}

\begin{lem}\label{lem:Delta-1}
    The map $\Delta^*_X$ descends to a homomorphism
    \[
      \Delta^*_X \colon \CH_0(X,D;A) \to \CH_0(S_X;S_A).
    \]
  \end{lem}
  \begin{proof}
We can reduce the proof of the lemma to the case of infinite base field
    using \cite[Proposition~6.1]{Binda-Krishna}. Note that the
    latter was stated when $A = \emptyset = D$ but its proof did not use these
    conditions.

    Assume now that $k$ is infinite. In this case, it suffices to
    show using \lemref{lem:Proj-bundle-trick} that $\Delta^*_X$
  descends to a homomorphism
  \begin{equation}\label{eqn:Delta-2}
    \Delta^*_X \colon \CH^{LW}_0(X,D;A) \to \CH^{LW}_0(S_X;S_A).
  \end{equation}

Let $C \subset X$ now be a Cartier curve relative to $D$ which misses $A$ and let
  $f \in k(C, C \cap D^\star)^{\times}$. By repeating the proof of
  \lemref{lem:reduction-basic} without using double
  (see \cite[Sublemma~1]{Biswas-Srinivas}),
  we can modify $C$ so as to assume that there is an integral surface
  $Y \subset X$ containing $C$ and not contained in $D$ such
that the following hold:
\begin{enumerate}
\item
  $Y \cap {X^o}$ is a complete intersection in ${X^o}$.
\item
  $C \subset Y$ is an integral Cartier divisor.
\item
  $Y \cap A^\star$ is finite.
  \item
    $C \cap A^\star = \emptyset$.
  \end{enumerate}

If follows from the above properties that $C$ is a Cartier curve everywhere
  on $X$. We let $C' = C \amalg_E C = S_C$, where $E := C \cap D$.
  Then $C' \subset S_X$ is a Cartier curve everywhere and  does not meet $S_{A^\star}$
  by \lemref{lem:Cartier}. Using the exact sequence (see ~\eqref{eqn:D*})
  \[
    0 \to \sO^{\times}_{C', C' \cap D} \to
    \sO^{\times}_{C, C \cap D} \times \sO^{\times}_{C, C \cap D} \to
    \sO^{\times}_{E} \to 0
    \]
    and that $C' \cap D' = C' \cap D$, we also get 
    $h = (f,f) \in k(C', C' \cap D')^{\times}$.
    Since $\Delta^*_X(\divf(f)) = \divf(h)$, we conclude that
    $\Delta^*_X(\divf(f))$ dies in $\CH^{LW}_0(S_X;S_A)$. This finishes the proof.
\end{proof}

\begin{lem}\label{lem:Proj-bundle-trick}
Let $k$ be an infinite field. Suppose that $\Delta^*_X$ descends to a homomorphism
    \[
      \Delta^*_X \colon \CH^{LW}_0(X,D;A) \to \CH^{LW}_0(S_X;S_A).
    \]
    for every $X \in \Sch_k$ satisfying the hypotheses of \lemref{lem:Delta-1}.
    Then $\Delta^*_X$ descends to a homomorphism
    \[
      \Delta^*_X \colon \CH_0(X,D;A) \to \CH_0(S_X;S_A).
    \]
    
  \end{lem}
  \begin{proof}
    Using \cite[Proposition~3.18]{Binda-Krishna}, the lemma is easily proven by
    mimicking the argument of the proof of \cite[Lemma~8.3]{Gupta-Krishna-BF}
    mutatis mutandis.
\end{proof}

\subsection{Proof of \thmref{thm:Main-1}}\label{sec:Main-1-pf}
We now prove \thmref{thm:Main-1}.
Let $(X,D)$ be as in the theorem and let $A \subset X$ be any reduced
closed subscheme such that $\dim(A) \le d-2$.
We show more generally that there is a split exact sequence
 \begin{equation}\label{eqn:Main-1-0}
    0 \to \CH_0(X|D;A^\star) \xrightarrow{p^+_*} \CH_0(S_X;S_A)
    \xrightarrow{\iota^*_-} \CH_0(X, D;A) \to 0.
  \end{equation} 
Taking $A = \emptyset$ proves \thmref{thm:Main-1}.
  Now, it is clear from the definitions of $p^\pm_*$ and $\iota^*_\pm$ that
~\eqref{eqn:Main-1-0} is a complex and $\iota^*_-$ is surjective.
  It follows from \propref{prop:lci-maps} that $p^+_*$ is split injective.
  It remains to show that ~\eqref{eqn:Main-1-0} is exact in the middle.
  But this easily follows from the fact that $\iota^*_\pm \circ \Delta^*_X$
  are the identity maps (see the proof of \cite[Theorem~7.1]{Binda-Krishna}).
$\hfill\square$

\subsection{Bloch's formula and moving lemma}\label{sec:BF*}
Let $k$ be any field and $X$ an integral quasi-projective surface over $k$
which is regular in codimension one. Let $D \subset X$ be a closed subscheme
such that $D^o = D \cap X^o$ is a Cartier divisor on $X^o$.
Let $\sK^M_{i, (X,D)}$ denote the kernel of the restriction map
$\sK^M_{i,X} \surj \sK^M_{i,D}$ of Zariski (and Nisnevich)
sheaves of Milnor $K$-groups on $X$.
By \cite[Lemma~3.2]{Gupta-Krishna-JAG} (see also \cite[Lemma~3.7]{Gupta-Krishna-REC}),
there is a cycle class map
$\cyc_{X} \colon \sZ_0(X^o \setminus D) \to H^2_\zar(X, \sK^M_{2,X})$ such that
the composite map
\begin{equation}\label{eqn:Milnor-0}
\sZ_0(X^o \setminus D) \xrightarrow{\cyc_{X}} H^2_\zar(X, \sK^M_{2,X})
\to H^2_\zar(X, \sK_{2,X}) \to K_0(X)
\end{equation}
takes a closed point $x \in X^o$ to the class $[\sO_{\{x\}}] \in K_0(X)$.
Here, $\sK_{i,X}$ is the Zariski sheaf of Quillen $K$-groups and
$H^2_\zar(X, \sK_{2,X}) \to K_0(X)$ is the edge map in the Thomason-Trobaugh
descent spectral sequence for algebraic $K$-theory (see \cite[Theorem~10.3]{TT}).
By \cite[Theorem~3.6]{BKS}, $\cyc_{X}$ factors through the map
\begin{equation}\label{eqn:Milnor-1}
  \cyc_{X} \colon \CH_0(X,D) \to H^2_\zar(X, \sK^M_{2,X}).
\end{equation}

We now look at the commutative diagram
\begin{equation}\label{eqn:Milnor-2}
  \xymatrix@C.8pc{
    \CH_0(X,D) \ar[r]^-{\cyc_{X}} \ar[d]_-{\Delta^*_X} &    H^2_\zar(X, \sK^M_{2,X})
      \ar[d]^-{\Delta^*_X} \\
      \CH_0(S_X) \ar[r]^-{\cyc_{S_X}} &  H^2_\zar(S_X, \sK^M_{2,S_X}).}
  \end{equation}

We saw in the proof of \thmref{thm:Main-1} that $\Delta^*_X$ on the left is
  injective. The map $\cyc_{S_X}$ is an isomorphism by
  \cite[Theorem~8.1]{BKS}. It follows that $\cyc_{X}$ is injective.
  This map is surjective by using \cite[Theorem~2.5]{Kato-Saito-86}
  in combination with \cite[\S~3.5]{Gupta-Krishna-REC}.
  It follows that $\cyc_X$ is an isomorphism. This proves the first isomorphism
  of \thmref{thm:Main-3}.
 $\hfill\square$

\vskip .3cm
  
 \begin{remk}\label{remk:Top}
  We remark that the above proof implies using \cite[Lemma~3.3]{BKS} that
  the change of topology map $H^2_\zar(X, \sK^M_{2,X}) \to H^2_\nis(X, \sK^M_{2,X})$
  is an isomorphism.
\end{remk}

\vskip .3cm

We now show that the cycle class map
  $\cyc_{X|D} \colon \sZ_0(X^o \setminus D) \to H^2_\zar(X, \sK^M_{2,(X,D)})$
  induces an isomorphism
  \begin{equation}\label{eqn:Milnor-3}
    \cyc_{X|D} \colon \CH_0(X|D) \xrightarrow{\cong} H^2_\zar(X, \sK^M_{2, (X,D)}).
  \end{equation}
  
    To this end, we look at the commutative diagram
    \begin{equation}\label{eqn:Milnor-4}
      \xymatrix@C.8pc{
        0 \ar[r] & \CH_0(X|D) \ar[r]^-{p^+_*} \ar@{.>}[d] &
        \CH_0(S_X) \ar[r]^-{\iota^*_-} \ar[d]^-{\cyc_{S_X}} &
        \CH_0(X,D) \ar[d]^-{\cyc_X} \ar[r] & 0 \\
        0 \ar[r] & H^2_\zar(S_X, \sK^M_{2, (S_X, X_-)}) \ar[r] &
        H^2_\zar(S_X, \sK^M_{2,S_X}) \ar[r]^-{\iota^*_-} & H^2_\zar(X, \sK^M_{2,X})
        \ar[r] & 0,}
      \end{equation}
      where the top sequence comes from \thmref{thm:Main-1} and the
      bottom sequence is split by the pull-back
      $\Delta^*_X \colon H^2_\zar(X, \sK^M_{2,X}) \to  H^2_\zar(S_X, \sK^M_{2,S_X})$.
      The map $\cyc_{S_X}$ is an isomorphism by \cite[Theorem~8.1]{BKS} and we showed
      above that $\cyc_X$ is also an isomorphism. It follows that $\cyc_{S_X}$
      induces an isomorphism
      \begin{equation}\label{eqn:Milnor-5}
        \cyc_{(S_X, X_-)} \colon \CH_0(X|D) \xrightarrow{\cong}
        H^2_\zar(S_X, \sK^M_{2, (S_X, X_-)}).
        \end{equation}

        It is straightforward to check that the composite map
        \[
    \sZ_0(X^o \setminus D)  \xrightarrow{\cyc_{(S_X, X_-)}}      
    H^2_\zar(S_X, \sK^M_{2, (S_X, X_-)}) \xrightarrow{\iota^*_+} H^2_\zar(X, \sK^M_{2, (X,D)})
  \]
  coincides with $\cyc_{X|D}$ of ~\eqref{eqn:Milnor-3}.
  To finish the proof therefore, we only have to show that the map
 \begin{equation}\label{eqn:Milnor-6} 
  \iota^*_+ \colon  H^2_\zar(S_X, \sK^M_{2, (S_X, X_-)})  \to 
  H^2_\zar(X, \sK^M_{2, (X,D)})
\end{equation}
is an isomorphism.
But this is proven in \cite[Lemma~13.3]{Binda-Krishna}, where we remark that
the latter result made an assumption that $X \in \Sm_k$ which was never used
in the proof.
This concludes the proof of the second isomorphism of \thmref{thm:Main-3}.
$\hfill\square$

\vskip .3cm

\begin{remk}\label{remk:Top-0}
  Using Remark~\ref{remk:Top} and the above proof, one easily checks that the
change of topology map
$H^2_\zar(X, \sK^M_{2,(X,D)}) \to H^2_\nis(X, \sK^M_{2,(X,D)})$ is also an isomorphism.
\end{remk}

\subsection{Proof of \corref{cor:Main-3-0}}\label{sec:ML-lci*}
We shall now prove \corref{cor:Main-3-0}. In the proof, we shall
show that part (3) of the corollary also holds if $k$ is algebraically closed
and $X$ is projective\footnote{This uses \cite{Levine-unpub} which is yet unpublished.
  One should therefore treat this addendum to the corollary as a conditional result.}.

We first assume that $X$ is a surface
over any field and note that $cyc_X$ has the factorization
\[
  \CH_0(X,D) \to \CH_0(X) \xrightarrow{\cyc_X} H^2_\zar(X, \sK^M_{2,X}).
  \]
  The second arrow is an isomorphism by \cite[Theorem~8.1]{BKS}
  and the composite arrow is the first isomorphism of \thmref{thm:Main-3}.
  The first arrow therefore is also an isomorphism.

  Suppose now that $k$ is algebraically closed and $d \ge 2$. We have a commutative diagram
  \begin{equation}\label{eqn:Main-3-0-0}
    \xymatrix@C.8pc{
      \CH^{LW}_0(X,D) \ar[r] \ar@{->>}[d] & \CH^{LW}_0(X) \ar@{->>}[d] \\
      \CH_0(X,D) \ar[r] & \CH_0(X).}
  \end{equation}
  The right vertical arrow is an isomorphism by \cite[Theorem~4.4]{Gupta-Krishna-JAG}
  (if $X$ is projective) and \cite[Corollary~7.6]{Krishna-Invent} (if $X$ is affine).
  It follows easily from \cite[Lemma~2.2]{Ghosh-Krishna-CFT} (see the proof of
  Lemma~2.3 of op. cit.) that the  top horizontal arrow is surjective.
  It suffices therefore to show that this arrow is injective too.

If $X$ is affine, then we have a factorization
  \[
    \CH^{LW}_0(X,D) \surj \CH^{LW}_0(X) \xrightarrow{\cyc_X} H^d_\nis(X, \sK^M_{d,X}),
  \]
  where the second arrow exists by \cite[Theorem~1.1]{Gupta-Krishna-JAG}.
  The second and the composite arrows can be shown to be isomorphisms
  by the same proof as in the surface case using
  \cite[Theorem~1.1]{Gupta-Krishna-JAG} instead of
  \cite[Theorem~8.1]{BKS}. The desired conclusion then follows.

  If $X$ is an integral projective scheme
  of dimension $d \ge 2$ over $k$ which is regular in codimension one, then
  we have a commutative diagram
  \[
   \xymatrix@C.8pc{
      \CH^{LW}_0(X,D) \ar[r] \ar[d]_-{\Delta^*_X} & K_0(X) \ar[d]^-{\Delta^*_X} \\
      \CH^{LW}_0(S_X) \ar[r] & K_0(S_X).}
  \]
  The left vertical arrow is injective by definition (see \secref{sec:Delta*}). 
Levine \cite{Levine-unpub}
  showed that the bottom horizontal  arrow is injective with rational coefficients.
  It follows that $\CH^{LW}_0(X,D) \to \CH^{LW}_0(X)$ has torsion kernel.
  We are therefore done by \lemref{lem:Torsion}.
 $\hfill\square$

 \vskip .3cm

\begin{lem}\label{lem:Torsion}
   Let $D \subset X$ be as in \corref{cor:Main-3-0} with $k$ algebraically closed
   and $X$ integral projective $R_1$-scheme. Then the canonical map
   \[
     \alpha_X \colon \CH^{LW}_0(X,D)_{\rm tor} \surj \CH^{LW}_0(X)_{\rm tor}
   \]
   is injective.
 \end{lem}
 \begin{proof}
   If $X$ is a surface, ~\eqref{eqn:Milnor-2} already shows that $\alpha_X$
   is bijective. We can thus assume that $\dim(X) \ge 3$.
 Let ${\delta} \in \sZ_0(X^o \setminus D)$ be a
0-cycle such that $[\delta] \neq 0$ in $\CH^{LW}_0(X,D)$, while
$n[\delta] \ = \ 0$ for some $n > 1$. We
will show that ${\alpha}_X ([\delta]) \neq 0$.

Our assumption implies that there
exists a finite collection of integral curves (see \cite[Lemma~1.3]{Levine-2})
$C_i$ contained in $X^o$ which are Cartier along $D$ and functions 
$f_i \in k(C_i, C_i \cap D)^{\times}$ such that $n\delta = {\sum} \divf(f_i)$.
Using Bloch's trick (see \cite[Lemma~5.2]{Bloch-Duke}), we let
$\pi : \wt X \to X$ be the successive blow-up at smooth points such
that if $\wt{C}_i$ is the strict transform of $C_i$ on $\wt X$, then the following hold:
\begin{enumerate}
\item
Each $\wt{C}_i$ is smooth;
\item
$\wt{C}_i$'s are pairwise disjoint;
\item
Each $\wt{C}_i$ meets the exceptional divisor of $\pi$ transversely
at smooth points.
\end{enumerate}

We let $\wt{D} = \pi^*(D)$.
Let ${\delta}'$ be a 0-cycle on $(\bigcup \wt{C}_i) \setminus \wt{D}$
such that ${\pi}_{*}
({\delta}') = \delta$. Now, $f_i$ can be considered as an
element of $k(\wt{C}_i, \wt{C}_i \cap \wt{D})^{\times}$
and we get ${\pi}_{*}(\eta) = 0$ if we let
$\eta =  n{\delta}' - {\sum} \divf (f_i)_{\wt{C}_i}$. 
Thus we can find smooth 
complete rational curves $L_j$ lying in the exceptional divisor,
and rational functions $g_j$ on $L_j$ such that $\eta = \sum
\divf(g_j)_{L_j}$. This gives $n{\delta}' = \sum \divf(f_i)_{\wt{C}_i} +
\sum \divf(g_j)_{L_j}$. We can
choose $L_j$'s in such a way that the curve
$C' = (\bigcup \wt{C}_i)  \cup  (\bigcup L_j)$
is reduced with smooth components and only ordinary double point singularities
(see the proof of \cite[Theorem~6.7]{Krishna-Invent}).
Since the support of $\eta$ is disjoint from $\wt{D}$, we can choose $L_j$'s such that
they lie in the components of the exceptional divisor which map to some closed
points of $X^o \setminus D$ under $\pi$. In particular, $(\bigcup L_j) \cap (\wt{D} \cup
\wt{X}_\sing) = \emptyset$.

We can now apply \cite[Corollary~3.15]{Ghosh-Krishna-Bertini} to find a
complete intersection integral surface $Z \subset \wt{X}$ containing $C'$ such that
$Z \cap \pi^{-1}(X^o)$ is regular and $Z \cap \wt{X}_\sing \cong Z \cap X_\sing$ is
  finite. Since $C' \subset Z_\reg$, it follows that each irreducible component of
  $C'$ is a Cartier divisor on $Z$.
  Letting $E = \wt{D} \cap Z$ and $A = \wt{X}_\sing \cap Z$,
  it follows that $\delta'$ is an element of
  $\sZ_0(Z_\reg \setminus (\wt{D} \cup A))$ such that $n \delta'$ dies in
  $\CH^{LW}_0(Z,E;A)$, hence in $\CH^{LW}_0(Z,E)$.
  We conclude from the surface case that
  $\alpha_Z(\delta') \neq 0$ in $\CH^{LW}_0(Z)$.
  Since $\CH^{LW}_0(Z;A) \xrightarrow{\cong} \CH^{LW}_0(Z)$ by
  \cite[Lemma~2.3]{Ghosh-Krishna-CFT}, it follows that
  $\alpha_Z(\delta') \neq 0$ in $\CH^{LW}_0(Z;A)$.

We now consider the commutative diagram 
\begin{equation}\label{eqn:Torsion-0}
  \xymatrix@C.8pc{
    \CH^{LW}_0(Z,E;A) \ar[r]^-{\alpha_Z} \ar[d] & \CH^{LW}_0(Z;A) \ar[r] &
    \CH^{LW}_0(\wt{X})
    \ar[d]^-{\pi_*} \\
    \CH^{LW}_0(X,D) \ar[rr]^-{\alpha_X} & & \CH^{LW}_0(X).}
\end{equation}

By a combination of \cite[Lemma~4.3]{Gupta-Krishna-JAG},
  the Roitman torsion theorem for normal projective varieties
  \cite[Theorem~1.6]{Krishna-Srinivas} and the hyperplane section theorem for
  the Albanese varieties of normal projective varieties
  \cite[Chapter~8, Theorem~5]{Lang}, we see that $\CH^{LW}_0(\wt{X})_{\rm tor}$
  satisfies the hyperplane section theorem.
  This implies that the top right horizontal arrow in ~\eqref{eqn:Torsion-0} is
  an isomorphism between the torsion subgroups. On the other hand, it is easy to check
  that $\pi_*$ is defined and is an isomorphism
  because $\pi$ is a composition of monoidal transformations with centers away from
  $X_\sing$. 
  We conclude that
  $\alpha_X([\delta]) \neq 0$. This finishes the proof.
\end{proof}

\vskip .3cm
  
  We end with the following question.

  \begin{ques}\label{ques:ML-LW-0}
    Let $X$ be an integral quasi-projective scheme of dimension $d \ge 2$
    over an infinite field which is regular in codimension one and let
    $Y \subset X$ be a nowhere dense closed subscheme.
     Is the canonical map
    \[
      \CH^{LW}_0(X,Y) \to \CH^{LW}_0(X)
    \]
    an isomorphism?
  \end{ques}

  If the above question has a positive answer, then using the projective bundle
  and pro-$\ell$-extension tricks we used above,
  one can show that the same holds for the lci
  Chow groups over any field. In this case, we can replace
  $\CH_0(X,D)$ by $\CH_0(X)$ in \thmref{thm:Main-1}.
  In particular, we can replace $\CH_0(X,D)$ by $\CH_0(X)$ in
  \thmref{thm:Main-1} when $d \le 2$ using \corref{cor:Main-3-0}.

\vskip .4cm

\noindent\emph{Acknowledgements.}
Gupta was supported by the
SFB 1085 \emph{Higher Invariants} (Universit\"at Regensburg).
Rathore would like to thank the Mathematics department of IISc, Bangalore for invitation,
where part of this work was done. The authors thank the referees for their valuable
comments.

\end{document}